\documentclass[twoside]{article}  

\usepackage[margin=2cm,bottom=2.5cm]{geometry} 
\usepackage{amsmath,amsthm,amssymb,latexsym}
\usepackage[v2,tips]{xy}

\newcommand{\vnten}{\overline\otimes}
\newcommand{\proten}{{\widehat{\otimes}}}
\newcommand{\lin}{\operatorname{lin}}
\newcommand{\mf}[1]{\mathfrak{#1}}
\newcommand{\mc}[1]{\mathcal{#1}}
\newcommand{\ip}[2]{{\langle {#1} , {#2} \rangle}}
\newcommand{\op}{{\operatorname{op}}}
\newcommand{\id}{\operatorname{id}}

\theoremstyle{plain}
\newtheorem{proposition}{Proposition}[section]
\newtheorem{theorem}[proposition]{Theorem}

\newtheorem{lemma}[proposition]{Lemma}
\theoremstyle{definition}

\begin{document}

\large
\title{Representing multipliers of the Fourier algebra on non-commutative $L^p$ spaces}
\author{Matthew Daws}
\maketitle

\begin{abstract}
We show that the multiplier algebra of the Fourier algebra on a locally compact group
$G$ can be isometrically represented on a direct sum on non-commutative $L^p$ spaces
associated to the right von Neumann algebra of $G$.  If these spaces are given their
canonical Operator space structure, then we get a completely isometric representation
of the completely bounded multiplier algebra.  We make a careful study of the
non-commutative $L^p$ spaces we construct, and show that they are completely isometric
to those considered recently by Forrest, Lee and Samei; we improve a result of theirs about
module homomorphisms.  We suggest a definition of a Figa-Talamanca--Herz algebra built
out of these non-commutative $L^p$ spaces, say $A_p(\hat G)$.  It is shown that
$A_2(\hat G)$ is isometric to $L^1(G)$, generalising the abelian situation.

\smallskip


\end{abstract}

\section{Introduction}

The Fourier algebra $A(G)$ is, for a locally compact group $G$, the space of
coefficient functionals $s\mapsto (\lambda(s)\xi|\eta)$ for $s\in G$, where
$\xi,\eta\in L^2(G)$.  Here
$\lambda$ denotes the left-regular representation of $G$ on $L^2(G)$.  For an
abelian group, $A(G)$ is nothing but the Fourier transform of $L^1(\hat G)$,
where $\hat G$ is the Pontryagin dual of $G$.  Eymard defined $A(G)$ for general
$G$ in \cite{eymard}.  We can also identify $A(G)$ as the predual of the group
von Neumann algebra $VN(G)$, see \cite[Chapter~VII, Section~3]{tak2}.

In this paper we shall be interested in the multiplier algebra of $A(G)$.  This
can either be thought of abstractly as the double centraliser algebra (see
\cite{johnson}) of $A(G)$, or, as $A(G)$ is a regular algebra of functions on
$G$, as the space of continuous functions $f$ such that $fa\in A(G)$ for
each $a\in A(G)$, see \cite{spronk} for example.  There is now much evidence
that $A(G)$ is often best viewed as an \emph{operator space}, when given the
standard operator space structure as the predual of $VN(G)$.  Then it is natural
to consider only the \emph{completely bounded} multipliers, leading to $M_{cb}A(G)$
(see \cite{spronk} or \cite{DcH}).  In \cite{neufang} a representation of
$M_{cb}A(G)$ on $\mc{CB}(B(L^2(G)))$ was defined, extending a representation of $M(G)$
defined in \cite{gha}.  It was shown that these representations are commutants of
each other, hence in some sense extending Pontryagin duality.  Similar ideas
were considered for Kac algebras in \cite{KR} and have been extended (along
with the commutation ideas) to Locally Compact Quantum Groups in \cite{jnr}.

In is shown in \cite{DcH} that both $MA(G)$ and $M_{cb}A(G)$ are dual spaces, in
such a way that the algebra products are separately weak$^*$-continuous (so these
are \emph{dual Banach algebras}); see also \cite[Section~6.2]{spronk}.  Now,
$\mc{CB}(B(L^2(G)))$ is also a dual Banach algebra, and the representation of
$M_{cb}A(G)$ constructed in \cite{neufang} is weak$^*$-weak$^*$-continuous.
However, it was shown in \cite[Corollary~3.8]{daws}
(and extended in \cite{uygul} to the completely bounded case) that a dual Banach
algebra $\mc A$ admits an isometric, weak$^*$-weak$^*$-continuous representation on
$\mc{B}(E)$ for some \emph{reflexive} Banach space $E$.  The space $E$ is built as the
large direct sum of real interpolation spaces, and is rather abstract.

In this paper, we shall show that we can represent $MA(G)$ on a direct sum of
non-commutative $L^p$ spaces associated to $VN(G)$; we can also represent $M_{cb}A(G)$
on the same space, if it is equipped with the canonical operator space structure.
Indeed, our construction is motivated by that of Young in \cite{young}; as Young didn't
consider multipliers, we sketch his ideas in Section~\ref{sec:first} below.

Once we have motivated looking at (non-commutative) $L^p$ spaces, we discuss weights
on $VN(G)$ and non-commutative $L^p$ for (possibly) non-semifinite von Neumann algebras
in Section~\ref{sec:second}.  This will involve introducing the complex interpolation method.
In Section~\ref{sec:third} we apply these ideas to the Fourier algebra, leading to a scale
of spaces $L^p(\hat G)$, for $1<p<\infty$, which are $A(G)$-modules.  We make a careful
study of these spaces, and prove some approximation results which allow us to work with
functions instead of abstract operators in the von Neumann algebra.  With this perspective,
the $A(G)$-module actions are just point-wise multiplication of functions.
We show that our spaces are (completely) isometrically isomorphic
to the two families of spaces constructed in \cite[Section~6]{fls}.  We think that our
construction is easier and more natural than that of \cite{fls}, although we have to
worry more about the details of the complex interpolation method.  The payoff is that,
for example, we can easily extend a cohomological result from \cite{fls}, which we can
show to hold for all values of $p$ (and not just $p\geq 2$).  

In Section~\ref{sec:rep_ag} we prove our representation result.  Let $p_n\rightarrow 1$
in $(1,\infty)$, and let $E$ be the $\ell^2$ direct sum of the spaces $L^{p_n}(\hat G)$.
Then $MA(G)$ is weak$^*$-weak$^*$-continuously isometric to the idealiser of $A(G)$ in
$\mc B(E)$.  If we equip $E$ with the canonical operator space structure, then $M_{cb}A(G)$
is weak$^*$-weak$^*$-continuously completely isometric to the idealiser of $A(G)$ in
$\mc{CB}(E)$.  As arguments involving multipliers often using bounded approximate identities,
it's worth stressing that our results hold for all locally compact groups $G$.  As hinted
at in Section~\ref{sec:first}, Figa-Talamanca--Herz algebras make a natural appearance, and
with our new tools, we define a notion of what $A_p(\hat G)$ should be for a non-abelian
group $G$.  We show that $A_2(\hat G)$ is canonically isometric to $L^1(G)$, but we have
been unable to decide if $A_p(\hat G)$ is always an algebra.

For Banach algebra notions, we follow \cite{dales} and \cite{palmer}; we always
write $E^*$ for the dual of a Banach or Operator space $E$, reserving the notation
$A'$ for the commutatant.  We shall only use standard facts about Operator spaces,
for which we refer the reader to \cite{ER} and \cite{pisier}.  In the few places where
we use matrix calculations, we shall simply write $\|\cdot\|$ for the norm on
$\mathbb M_n(E)$, for any $n$.

\section{Group convolution algebras}\label{sec:first}

In this section we quickly review Young's construction in \cite[Theorem~4]{young}, as
applied to multipliers.  Let $G$ be a locally compact group, and consider the group
convolution algebra $L^1(G)$.  The multiplier algebra of $L^1(G)$ can be isometrically
isomorphically identified with $M(G)$, the measure algebra of $G$.  This is Wendel's
theorem, \cite{wendel} or \cite[Theorem~3.3.40]{dales}.

Let $(p_n)$ be some sequence in $(1,\infty)$ converging to $1$.  Let $E$ be the direct
sum, in an $\ell^2$ sense, of the spaces $L^{p_n}(G)$.  To be exact, $E$ consists of
sequences $(\xi_n)$ where, for each $n$, $\xi_n\in L^{p_n}(G)$, with
\[ \|(\xi_n)\| := \Big( \sum_n \|\xi_n\|^2_{p_n} \Big)^{1/2} < \infty. \]
Thus $E$ is reflexive.
Then $M(G)$ acts contractively on each $L^{p_n}(G)$ space by convolution, and hence also
on $E$, leading to a contractive homomorphism $\theta:M(G)\rightarrow \mc B(E)$.

\begin{theorem}\label{thm:for_measure}
With notation as above, $\theta$ is isometric and weak$^*$-weak$^*$-continuous.
\end{theorem}

We first introduce some further concepts.  We write $\proten$ for the (completed)
projective tensor product (see \cite[Appendix~A3]{dales} for example).  For any
reflexive Banach space $F$, we thus have that $\mc B(F) = (F \proten F^*)^*$.
Let $\lambda_p:L^1(G)\rightarrow\mc B(L^p(G))$ be the left-regular representation, and
let $(\lambda_p)_*:L^p(G)\proten L^{p'}(G) \rightarrow L^\infty(G)$ be the
adjoint.  Here $p'$ is the conjugate index to $p$, so that $L^p(G)^* = L^{p'}(G)$.  For
$a\in L^1(G), \xi\in L^p(G)$ and $\eta\in L^{p'}(G)$, we see that
\begin{align*} \ip{(\lambda_p)_*(\xi\otimes\eta)}{a} &= \ip{\eta}{\lambda_p(a)(\xi)}
= \int_G \int_G \eta(t) a(s) \xi(s^{-1}t) \ ds \ dt
= \ip{\omega_{\xi,\eta}}{a}.
\end{align*}
Here $\omega_{\xi,\eta}$ denotes the function $s\mapsto \int_G \xi(s^{-1}t)\eta(t) \ dt$.
Thus $\omega_{\xi,\eta}$ is a member of the Figa-Talamanca--Herz algebra $A_p(G)$,
identified as a subalgebra of $C_0(G)\subseteq L^\infty(G)$.  For further details see
\cite{herz1,herz2}.

This then suggests an abstract way to define $\tilde\theta:M(G)\rightarrow\mc B(L^p(G))$,
namely
\[ \ip{\eta}{\tilde\theta(\mu)(\xi)} = \ip{\mu}{\omega_{\xi,\eta}}
\qquad (\mu\in M(G), \xi\in L^p(G), \eta\in L^{p'}(G)). \]
By the above calculation, this extends $\theta$.  Furthermore, if $\xi,\eta\in C_{00}(G)$,
the space of compactly support continuous functions, then $\xi\in L^p(G)$, $\eta\in
L^{p'}(G)$, and for $\mu\in M(G)$ we see that
\[ \ip{\eta}{\tilde\theta(\mu)(\xi)} = \int_G \int_G \xi(s^{-1}t) \eta(t) \ dt
\ d\mu(s) = \ip{\eta}{\mu*\xi}, \]
where $\mu*\xi$ has the unambiguous meaning of $\mu$ convolved with $\xi$.  As such
$\xi$ and $\eta$ are dense, we are justified in saying that $\tilde\theta$ is simply
the convolution action of $M(G)$ on $L^p(G)$.

\begin{proof}[Proof of Theorem~\ref{thm:for_measure}]
Consider the adjoint map $\theta_*:E\proten E^*\rightarrow M(G)^*$ given by
\[ \ip{\theta_*(\xi\otimes\eta)}{\mu} = \ip{\eta}{\theta(\mu)(\xi)}
= \sum_n \ip{\eta_n}{\tilde\theta(\mu)(\xi_n)}
= \sum_n \ip{\mu}{\omega_{\xi_n,\eta_n}}, \]
where $\xi = (\xi_n)\in E, \eta=(\eta_n)\in E^*$ and $\mu\in M(G)$.  In particular,
$\theta_*$ maps into $C_0(G)$, the predual of $M(G)$, so that $\theta$ is
weak$^*$-weak$^*$-continuous.

For $f,g\in C_{00}(G)$, we have that $\omega_{f,g} = g * \check f$ as functions,
where $\check f(s) = f(s^{-1})$ for $s\in G$.  Furthermore, we have that
\[ \lim_{p\rightarrow 1} \|f\|_p = \|f\|_1, \quad \lim_{p'\rightarrow\infty}
\|g\|_{p'} = \|g\|_\infty. \]
For any $g\in C_{00}(G)$ and $\epsilon>0$, we can find some $f\in C_{00}(G)$
with $\|f\|_1=1$ and $\|g*\check f - g\|_\infty < \epsilon$ (for example, see the proof
of \cite[Lemma~3.3.22]{dales}).  As $p_n\rightarrow 1$, we can find $n$ with
$\|g\|_{p_n'} < (1+\epsilon)\|g\|_\infty$ and $\|f\|_{p_n} < 1+\epsilon$.  It
follows that
\[ |\ip{\mu}{g}| \geq |\ip{\mu}{\omega_{f, g}}| - \epsilon\|\mu\|, \]
and that
\[ \|\omega_{f,g}\|_{A_{p_n}(G)} \leq \|f\|_{p_n} \|g\|_{p_n'}
< (1+\epsilon)^2 \|g\|_\infty. \]
By taking suitable supremums, it now follows easily that $\theta$ is an isometry.
\end{proof}

For a Banach algebra $\mc A$, we say that $\mc A$ is \emph{faithful} if
for $a\in\mc A$, when $bac=0$ for all $b,c\in\mc A$, then $a=0$.  We shall always
assume that our algebras are faithful: notice that if $\mc A$ is unital, or has
an approximate identity, then $\mc A$ is faithful.
A pair $(L,R)$ of linear maps $A\rightarrow A$ is a \emph{multiplier}
(or \emph{centraliser}) if
\[ L(ab) = L(a)b, \quad R(ab) = aR(b), \quad
aL(b) = R(a)b \qquad (a,b\in\mc A). \]
The Closed Graph Theorem then shows that $L,R\in\mc B(\mc A)$.  For further
details see \cite{johnson}, \cite{dales} or \cite[Section~1.2]{palmer}.
Indeed, \cite[Theorem~1.2.4]{palmer} shows that if $L,R:\mc A\rightarrow\mc A$ are
any maps with $aL(b) = R(a)b$ for $a,b\in\mc A$,
then $(L,R)$ is already a multiplier.  Let $M(\mc A)$ be the space of multipliers,
normed by embedding into $\mc B(\mc A)\times\mc B(\mc A)$, and made into an algebra
for the product $(L,R)(L',R') = (LL', R'R)$.  Notice that $\mc A$ embeds (as $\mc A$
faithful) into $M(\mc A)$ by $a\mapsto (L_a,R_a)$ where $L_a(b)=ab, R_a(b)=ba$ for
$a,b\in\mc A$.

Then Wendel's Theorem tells us that for $(L,R)\in M(L^1(G))$ there exists a unique
$\mu\in M(G)$ such that $L(a) = \mu a$ and $R(a) = a\mu$ for $a\in L^1(G)$.  Indeed,
from the proof of \cite[Theorem~3.3.40]{dales}, we have that $\mu$ is the weak$^*$-limit
of $(L(e_\alpha))$ in $M(G)$, where $(e_\alpha)$ is a bounded approximate identity
for $L^1(G)$.  It is then easy to show that $L(a) = \mu a$ for $a\in L^1(G)$.  Notice
then that $R(a)b = aL(b) = a (\mu b) = (a\mu) b$ for $a,b\in L^1(G)$, so as $L^1(G)$
is faithful, $R(a) = a\mu$ as required.

\begin{theorem}\label{thm:ideal_meas}
With notation as above, the image of $\tilde\theta:M(G)\rightarrow\mc B(E)$ is
exactly the \emph{idealiser} of $\theta(L^1(G))$, namely
\[ \mc I = \{ T\in\mc B(E) : T\theta(a), \theta(a)T\in\theta(L^1(G)) \ (a\in L^1(G)) \}. \]
\end{theorem}
\begin{proof}
For $\mu\in M(G)$, we have $\tilde\theta(\mu)\theta(a) = \theta(\mu a)$
and $\theta(a) \tilde\theta(\mu) = \theta(a \mu)$ for $a\in L^1(G)$, so that
$\tilde\theta(M(G)) \subseteq \mc I$.

Conversely, let $T\in\mc I$ and define $L,R:L^1(G)\rightarrow L^1(G)$ by
\[ L(a) = \theta^{-1}\big( T \theta(a) \big), \quad
R(a) = \theta^{-1}\big( \theta(a) T \big) \qquad (a\in L^1(G)), \]
which makes sense, as $\theta$ is injective onto its range.  For $a,b\in L^1(G)$
we see that $\theta(a) \theta(L(b)) = \theta(a) T\theta(b) = \theta(R(a)) \theta(b)$,
so that $aL(b) = R(a)b$.  Thus $(L,R)\in M(L^1(G))$.  Hence there exists $\mu\in M(G)$
with $L(a) = \mu a$ for $a\in L^1(G)$, so that $\tilde\theta(\mu)\theta(a)
= T\theta(a)$ for $a\in L^1(G)$.

By the construction of $E$, we see that $\{ \theta(a)\xi : a\in L^1(G), \xi\in E \}$
is linearly dense in $E$, from which it follows that $T = \tilde\theta(\mu)$,
completing the proof.
\end{proof}

Notice that we implicitly used the Closed Graph Theorem, in invoking
\cite[Theorem~1.2.4]{palmer}.  In the completely bounded setting, this would not be
available to us, and indeed, it is unclear to the author if a direct analogue of
this result would be true.  However, if $\mc A$ is commutative (or has a bounded
approximate identity) that $L$ and $R$ are closely related, allowing a modification
of the proof to work, see Theorem~\ref{thm:mcb_rep} below.  In relation to this, it is
interesting to note that \cite{jnr} works with \emph{one-sided} multipliers (or centralisers).

It is classical that $L^p(G)$ can be described as a complex interpolation space
between $L^1(G)$ and $L^\infty(G)$; see below for definitions, or \cite[Chapter~4]{bergh}.
We can recover the action of $L^1(G)$ on $L^p(G)$ by interpolation, but some care
is needed.  Indeed, obviously $L^1(G)$ is an $L^1(G)$-bimodule over itself, and so
by duality, $L^\infty(G)$ is an $L^1(G)$-bimodule.  However, notice that the resulting
left action of $L^1(G)$ on $L^\infty(G)$ is \emph{not} the usual convolution action.
With this in mind, the constructions in Section~\ref{sec:third} below should
appear less artificial.

\section{Non-commutative $L^p$ spaces}\label{sec:second}

In this section we sketch the complex interpolation approach to non-commutative
$L^p$ spaces, see \cite{terp} and \cite{izumi}.

\subsection{Weights on group von Neumann algebras}

For a locally compact group $G$, let $\lambda$ and $\rho$ be, respectively,
the left- and right-regular representations, defined by
\[ \big(\lambda(s)\xi\big)(t) = \xi(s^{-1}t), \quad
\big(\rho(s)\xi\big)(t) = \rho(ts) \nabla(s)^{1/2} \qquad (\xi\in L^2(G),
s,t\in G). \]
Here $\nabla$ is the modular function on $G$.  For $f\in L^1(G)$, we shall
write $\lambda(f)$ and $\rho(f)$ for the operators induced by integration, for
example
\[\big(\rho(f)\xi\big)(s) = \int_G f(t) \xi(st) \nabla(t)^{1/2} \ dt
\qquad (\xi\in L^2(G)). \]

Then the group von Neumann algebra
$VN(G)$ is the von Neumann algebra generated by $\lambda$, so $VN(G) = \lambda(G)''$.
Similarly, the right group von Neumann algebra, denoted here by $VN_r(G)$, is
generated by $\rho$.  We have that $VN(G)'=VN_r(G)$ and $VN_r(G)'=VN(G)$, see
\cite[Chapter~VII, Section~3]{tak2}.

An alternative way to construct $VN(G)$ is to start with $C_{00}(G)$, considered
as a left Hilbert algebra.  The inner-product is inherited from $L^2(G)$, the
product is convolution, and the involution is $f^\sharp(s) = \overline{f(s^{-1})}
\nabla(s)^{-1}$ for $f\in C_{00}(G), s\in G$.  See \cite{tak2} or \cite{sz}
for further details on left Hilbert algebras.  One word of caution: for
$f\in C_{00}(G)$ (or more generally, for \emph{right bounded} elements of
$L^2(G)$) we can define $\pi_r(f)\in VN_r(G)$ (using the notation of \cite{tak2}).
This is \emph{not} equal to $\rho(f)$; we have $\pi_r(f) = \rho(K(f))$ for $K$
defined below.

At this point, we shall stress that henceforth, for functions $a,b$ on $G$,
we denote the convolution product by $ab$ (when this makes sense)
and the point-wise product by $a\cdot b$.
An exception is that $\nabla$ always acts by point-wise multiplication.

The left Hilbert algebra leads naturally to a weight $\varphi$ on $VN(G)$.  This
weight is explored in detail by Haagerup in \cite[Section~2]{haa}.
We let $\mf n_\varphi = \{ x\in VN(G) : \varphi(x^*x)<\infty \}$ and
$\mf m_\varphi = \lin \mf n_\varphi^* \mf n_\varphi$, and extend $\varphi$ to
$\mf m_\varphi$ in the usual way.  Let us just note that
\[ \varphi(\lambda(f)) = f(e_G) \qquad (f\in C_{00}(G)^2), \]
where $e_G$ is the unit of $G$, and $C_{00}(G)^2 = \lin \{ fg : f,g\in C_{00}(G) \}$.

Let $(\pi, H, \Lambda)$ be the GNS construction for $(VN(G),\varphi)$.  We may
hence identify $H$ with $L^2(G)$ by $\Lambda(\lambda(f)) = f$ for $f\in C_{00}(G)$
(or more generally for left bounded $f\in L^2(G)$).  Henceforth we shall drop $\pi$
and always regard $VN(G)$ as acting on $L^2(G)$.
The \emph{modular conjugation} is the map
\[ J:L^2(G)\rightarrow L^2(G), \quad J\xi(s) = \overline{\xi(s^{-1})} \nabla(s)^{-1/2}
\qquad (\xi\in L^2(G), s\in G). \]
We define a linear version of $J$ to be $K$, where $K(\xi) = J(\overline{\xi})$ for
$\xi\in L^2(G)$.  We define the ``check map'' by $\check\xi(s) =\xi(s^{-1})$,
so $K\xi = \check\xi \nabla^{-1/2}$.
We have that $VN_r(G) = VN(G)' = JVN(G)J$, and
\[ \lambda(f) = J \rho(\overline{f}) J = K \rho(f) K \qquad (f\in L^1(G)). \]
The \emph{modular operator} is given by point-wise multiplication
by $\nabla$, and this leads to the \emph{modular automorphism group}
$(\sigma_t)_{t\in\mathbb R}$ given by $\sigma_t(\cdot) = \nabla^{it} (\cdot) \nabla^{-it}$.

We shall pick a canonical choice of weight $\varphi'$ on $VN_r(G)$ by
\[ \varphi'(x) = \varphi(Jx^*J) \qquad (x\in VN_r(G)^+). \]
Then $\mf m_{\varphi'} = J \mf m_\varphi J$ and the formula above defines $\varphi'$
on $\mf m_{\varphi'}$.  Let $(\pi', H', \Lambda')$ be the GNS construction for $\varphi'$.
We can identify $H'$ with $H$ by $\Lambda'(x) = J\Lambda(JxJ)$ for $x\in \mf n_{\varphi'}
= J \mf n_\varphi J$.  Hence we identify $H'$ with $L^2(G)$ by
\[ \Lambda'(\rho(f)) = J\Lambda(\lambda(\overline{f})) = K(f)
\qquad (f\in C_{00}(G)). \]
Again, we suppress $\pi'$ and regard $VN_r(G)$ as acting on $L^2(G)$.  Then the modular
conjugation for $\varphi'$ is simply $J$.  The modular automorphism group for
$\varphi'$ is $(\sigma'_t)_{t\in\mathbb R}$, and this is given by $\sigma'_t(x)
= J\sigma_t(JxJ)J$ for $x\in VN_r(G)$.  Some care is required when analytically
extending this to complex values; indeed, we have
$\sigma'_z(x) = J\sigma_{\overline z}(JxJ)J$ for analytic $x$ and $z\in\mathbb C$.
Consequently
\[ \sigma'_z(\rho(f)) = \rho\big( \nabla^{-i\overline z} f \big)
\qquad (f\in C_{00}(G), z\in\mathbb C). \]

\subsection{Non-commutative $L^p$ spaces}

There is a long history to non-commutative $L^p$ spaces, for which we refer the reader
to \cite{PZ}.  For a von Neumann algebra $\mc M$ with a finite normal trace
$\tau$, we can simply let $L^p(\mc M,\tau)$ be the completion of $\mc M$ with
respect to the norm $\|x\|_p = \tau((x^*x)^{p/2})$, for $1\leq p<\infty$.  Similar
remarks apply to semi-finite traces, although the framework of ``measurable operators''
gives a realisation of the completed space.  See \cite[Chapter~IX, Section~2]{tak2}
for further details.

For a general von Neumann algebra which might only admit a weight, Haagerup introduced
a crossed-product construction of a non-commutative $L^p$ space in \cite{haa_noncomm}.
Building on work of Connes, Hilsum provided a spatial definition of a non-commutative
$L^p$ space in \cite{hilsum}, and showed that the resulting space was isometrically
isomorphic to Haagerup's.  By analogy with the commutative case, we might expect the
complex interpolation method to play a role.  In \cite{kosaki}, Kosaki provided a
construction of a non-commutative $L^p$ space associated to a von Neumann algebra with
a finite weight (that is, a normal state) using the complex interpolation method.
He showed that his space is isometrically isomorphic to Haagerup's.  In \cite{terp}, Terp
extended a special case of Kosaki's construction to the semi-finite case, and she showed
that her $L^p$ space is isometrically isomorphic to Hilsum's (and hence to Haagerup's).

We shall instead follow Izumi's construction in \cite{izumi}, which simultaneously
generalises Kosaki's and Terp's constructions.  Of particular interest is that in
\cite{izumi2}, Izumi makes a detailed study of his spaces, introducing bilinear and
sesquilinear products, and showing that his $L^2$ spaces are canonically isometrically
isomorphic to the standard Hilbert space constructed from the underlying weight.
As such, Izumi's constructions are self-contained (although we note that,
technically, he relies upon Terp's work in a proof in \cite{izumi}).

First let us define the complex interpolation method.  See \cite{bergh},
\cite[Section~2.7]{pisier} for further details.  A \emph{compatible couple} of
Banach spaces is a pair $(E_0,E_1)$ continuously embedded into a Hausdorff topological
vector space $X$.  We can then make sense of the spaces $E_0\cap E_1$ and $E_0+E_1$,
and define norms on them by
\begin{gather*} \|x\| = \max( \|x\|_{E_0}, \|x\|_{E_1} ) \qquad (x\in E_0\cap E_1), \\
\|x\| = \inf\{ \|a\|_{E_0}+\|b\|_{E_1} : x=a+b, a\in E_0, b\in E_1\}
\qquad (x\in E_0+E_1).
\end{gather*}
We need $X$ to be Hausdorff to ensure that we get a norm on $E_0+E_1$.  However,
once we can form $E_0+E_1$, in what follows, we can always just replace $X$ by
$E_0+E_1$.

Let $\mc S = \{ z=x+iy\in\mathbb C : 0\leq x\leq 1, y\in\mathbb R\}$ and
$\mc S_0 = \{ z=x+iy\in\mathbb C : 0< x< 1, y\in\mathbb R\}$.
We let $\mc F$ be the space of functions $f:\mc S\rightarrow E_0+E_1$ such that:
\begin{enumerate}
\item $f$ is continuous and bounded, and analytic on $\mc S_0$;
\item for $j=0,1$, we have that $\mathbb R\mapsto E_j; t\mapsto f(j+it)$
is continuous, bounded, and tends to $0$ as $|t|\rightarrow\infty$.
\end{enumerate}
For more details on vector-valued analytic functions, see \cite[Appendix]{tak2}
for example.  We give $\mc F$ a norm by setting
\[ \|f\| = \max_{j=0,1} \sup_{t\in\mathbb R} \|f(j+it)\|_{E_0}
\qquad (f\in\mc F). \]
This is a norm, and then $\mc F$ becomes a Banach space.

For $0\leq\theta\leq 1$, we define $(E_0,E_1)_{[\theta]}$ to be the subspace of
$E_0+E_1$ consisting of those $x$ such that $x=f(\theta)$ for some $f\in\mc F$,
together with the quotient norm
\[ \|x\|_{[\theta]} = \inf\{ \|f\| : f\in\mc F, f(\theta)=x \}. \]
The following is proved in \cite[Theorem~4.1.2]{bergh}.

\begin{theorem}\label{thm:inter}
With notation as above, we have norm decreasing inclusions $E_0\cap E_1
\rightarrow (E_0,E_1)_{[\theta]} \rightarrow E_0+E_1$.  Let $(F_0,F_1)$ be
another pair of compatible Banach spaces, and let $T:E_0+E_1\rightarrow F_0+F_1$
be a linear map such that for $j=0,1$, $T(E_j) \subseteq F_j$ and the restriction
$T:E_j\rightarrow F_j$ is bounded.  Then
\[ T\big( (E_0,E_1)_{[\theta]} \big) \subseteq (F_0,F_1)_{[\theta]},
\quad \|T\| \leq \|T:E_0\rightarrow F_0\|^{1-\theta} \|T:E_1\rightarrow F_1\|^\theta. \]
\end{theorem}

\begin{lemma}\label{lem:when_compat}
With notation as above, for $j=0,1$ let $T_j\in\mc B(E_j,F_j)$.  There exists
$T:E_0+E_1\rightarrow F_0+F_1$ with $T|_{E_j} = T_j$ for $j=0,1$ if, and only if,
$T_0$ and $T_1$ map $E_0\cap E_1$ into $F_0\cap F_1$ and agree on $E_0\cap E_1$.
\end{lemma}
\begin{proof}
If $T_0$ and $T_1$ agree on $E_0\cap E_1$ and map into $F_0\cap F_1$,
then we try to define $T$ by $T(x_0+x_1) = T_0(x_0)+T_1(x_1)$ for $x_0\in E_0,
x_1\in E_1$.  This is well-defined, for if $x_0+x_1 = x_0'+x_1'$ then $x_0-x_0'
= x_1'-x_1\in E_0\cap E_1$ and so $T_0(x_0)-T_0(x_0') = T_1(x_1')-T_1(x_1)
\in F_0\cap F_1$.  The converse is clear.
\end{proof}

There is also a bilinear version, see \cite[Theorem~4.4.1]{bergh}.

\begin{theorem}\label{thm:bilin_int}
Let $(E_0,E_1)$, $(F_0,F_1)$ and $(G_0,G_1)$ be compatible couples, and let
$T:E_0\cap E_1 \times F_0\cap F_1 \rightarrow G_0\cap G_1$ be a bilinear map such
that for some constants $M_0,M_1$, we have
\[ \|T(x_j,y_j)\|_{G_j} \leq M_j \|x_j\|_{E_j} \|y_j\|_{F_j}
\qquad (j=0,1, x_j\in E_j, y_j\in F_j). \]
For $0<\theta<1$, there is a bilinear map
\[ T_\theta:(E_0,E_1)_{[\theta]} \times (F_0,F_1)_{[\theta]}
\rightarrow (G_0,G_1)_{[\theta]}, \]
which extends $T$, and which is bounded by $M_0^{1-\theta} M_1^\theta$.
\end{theorem}

Now let $\mc M$ be a von Neumann algebra with normal semi-finite weight $\varphi$.
Let $(H,\Lambda)$ be the GNS construction, where we identify $\mc M$ with a subalgebra
of $\mc B(H)$.  Let $J$ be the modular conjugation, and $\nabla$ the modular operator.
We shall now sketch Izumi's approach to non-commutative $L^p$ spaces.  The idea is to
turn $(\mc M,\mc M_*)$ into a compatible couple; then $L^p(\varphi)$ will be defined as
$(\mc M,\mc M_*)_{[1/p]}$, for $1<p<\infty$.

Let $(H,\Lambda)$ be a GNS construction for $\varphi$, so that $\mf A=\Lambda(\mf n_\varphi
\cap \mf n_\varphi^*)$ is a full left Hilbert algebra in $H$, which generates $\mc M$
and induces $\varphi$.  Let $\mf A_0$ be the maximal Tomita algebra associated to $\mf A$,
see \cite[Chapter~VI, Section~2]{tak2}, and let $\mf a_0 = \Lambda^{-1}(\mf A_0)$.  In particular, each
$x\in\mf a_0$ is analytic for $(\sigma_t)$, and $\Lambda(x)$ is in the domain of $\nabla^\alpha$
for each $\alpha\in\mathbb C$.

For $\alpha\in\mathbb C$, we let $L_{(\alpha)}$ be the collection of those $x\in\mc M$ such
that there exists $\varphi^{(\alpha)}_x\in\mc M_*$ with $\ip{y^*z}{\varphi^{(\alpha)}_x} =
(xJ\nabla^{\overline\alpha}\Lambda(y)|J\nabla^{-\alpha}\Lambda(z))$ for $y,z\in\mf a_0$.
Then $L_{(\alpha)}$ is a subspace of $\mc M$ which contains $\mf a_0^2$, and is hence
$\sigma$-weakly dense.  We norm $L_{(\alpha)}$ by setting $\|x\|_{L_{(\alpha)}} = \max(
\|x\|_{\mc M}, \|\varphi^{(\alpha)}_x\|_{\mc M_*} )$ for $x\in L_{(\alpha)}$.  Let
$i_{(\alpha)}:L_{(\alpha)}\rightarrow\mc M$ be the inclusion map, and let $j_{(\alpha)}:
L_{(\alpha)}\rightarrow\mc M_*$ be the map $x\mapsto \varphi^{(\alpha)}_x$.  These are
contractive injections, and $j_{(\alpha)}$ has norm dense range.  Izumi proves that we
have the following commuting diagram
\[ \xymatrix{ & \mc M \ar[rd]^{j^*_{(-\alpha)}} \\
L_{(\alpha)} \ar[ru]^{i_{(\alpha)}} \ar[rd]_{j_{(\alpha)}} & & L^*_{(-\alpha)} \\
& \mc M_* \ar[ru]_{i^*_{(-\alpha)}} } \]
In particular, we have that $\ip{y}{\varphi_x^{(\alpha)}} = \ip{x}{\varphi_y^{(-\alpha)}}$
for $x\in L_{(\alpha)}$ and $y\in L_{(-\alpha)}$.  By density we have that $i^*_{(-\alpha)}$
and $j^*_{(-\alpha)}$ are injective, and so we can view $(\mc M,\mc M_*)$ as a compatible
couple.  Izumi shows that under this identification, $\mc M\cap\mc M_*$ is precisely $L_{(\alpha)}$.
We finally set
\[ L^p_{(\alpha)}(\varphi) = \big(\mc M,\mc M_*\big)_{[1/p]} \qquad (1<p<\infty). \]
We shall always view $L^p_{(\alpha)}(\varphi)$ as a subspace of $\mc M+\mc M_*$;
consequently, by the commuting diagram and Theorem~\ref{thm:inter}, we have that
$j^*_{(-\alpha)}(x)\in L^p_{(\alpha)}(\varphi)$ for all $x\in L$ and all $p$.

For most of this paper, we shall actually work just with the case $\alpha=0$, which is
exactly the case which Terp considers in \cite{terp}.  Set $L = L_{(0)}$, so we actually
have the stronger property that $x\in L$ when there exists $\varphi_x\in\mc M_*$ with
\[ \ip{y^*z}{\varphi_x} = (Jx^*J\Lambda(z)|\Lambda(y))
= (xJ\Lambda(y)|J\Lambda(z)) \qquad (y,z\in\mf n_\varphi). \]

As shown in \cite{izumi2}, there are bilinear maps which satisfy
\[ \ip{\cdot}{\cdot}_{p,(\alpha)} : L^p_{(\alpha)}(\varphi) \times L^{p'}_{(-\alpha)}(\varphi)
\rightarrow\mathbb C; \quad \ip{j^*_{(-\alpha)}(x)}{j^*_{(\alpha)}(y)}
= \ip{y}{\varphi_x^{(\alpha)}} = \ip{x}{\varphi^{(-\alpha)}_y}, \]
where $1/p + 1/p'=1$.  There are sesquilinear maps which satisfy
\[ (\cdot|\cdot)_{p,(\alpha)} : L^p_{(\alpha)}(\varphi) \times
L^{p'}_{(\overline\alpha)}(\varphi) \rightarrow\mathbb C; \quad
( j^*_{(-\alpha)}(x) | j^*_{(-\overline\alpha)}(y) )_{p,(\alpha)}
= \ip{y^*}{\varphi_x^{(\alpha)}} = \overline{\ip{x^*}{\varphi^{(\overline\alpha)}_x}}. \]
Furthermore, these maps implement dualities between $L^p_{(\alpha)}(\varphi)$ and
$L^{p'}_{(-\alpha)}(\varphi)$, and between $L^p_{(\alpha)}(\varphi)$ and
$L^{p'}_{(\overline\alpha)}(\varphi)$, respectively.  As such, the dual of
$L^p_{(0)}(\varphi)$ can be identified with $L^{p'}_{(0)}(\varphi)$,
both linearly and anti-linearly.

We can identify $L^2_{(-1/2)}(\varphi)$ with $H_\varphi$, the standard GNS space for
$\varphi$.  Indeed, there is an isometric isomorphism
\[ h:H_\varphi \rightarrow L^2_{(-1/2)}(\varphi); \quad
h(\Lambda(x)) = j^*_{(-1/2)}(x) \qquad (x\in\mf n_\varphi). \]
Furthermore, $h$ respects the relevant inner-products, that is
\[ (\xi|\eta) = (h(\xi)|h(\eta))_{2,(-1/2)} \qquad (\xi,\eta\in H_\varphi). \]
We can translate this to other values of $\alpha$ by using the fact that there are
isometric isomorphisms
\[ U_{p,(\beta,\alpha)}: L^p_{(\alpha)}(\varphi) \rightarrow 
L^p_{(\beta)}(\varphi); \quad U_{p,(\beta,\alpha)}(j^*_{(-\alpha)}(x))
= j^*_{(-\beta)}(\sigma_{i(\beta-\alpha)/p}(x))
\qquad (x\in\mf a_0^2, \alpha,\beta\in\mathbb R). \]
Then, again for $\alpha,\beta\in\mathbb R$, we have that
\[ ( U_{p,(\beta,\alpha)}(\xi) | U_{p,(\beta,\alpha)}(\eta) )_{p,(\beta)}
= (\xi|\eta)_{p,(\alpha)} \qquad (\xi,\eta\in L^p_{(\alpha)}(\varphi)). \]
In particular, there is an isometric isomorphism
$k:H_\varphi\rightarrow L^2_{(0)}(\varphi)$ with
\[ k(\Lambda(x)) = j^*_{(0)}(\sigma_{i/4}(x)) \quad (x\in\mf a_0^2), \qquad
(\xi|\eta) = (k(\xi)|k(\eta))_{2,(0)} \quad (\xi,\eta\in H_\varphi). \]

Using convergence theorems for integration, it is easy to show that if $(X,\mu)$ is a
measure space, and $f\in L^1(\mu)\cap L^\infty(\mu)$, then $f\in L^p(\mu)$ for all
$p\in(1,\infty)$, and $\lim_{p\rightarrow 1} \|f\|_p = \|f\|_1$.  The following is
a non-commutative version of this.

\begin{proposition}\label{prop:limit}
With notation as above, let $x\in L$.  Then $\lim_{p\rightarrow 1} \|j^*_{(0)}(x)\|_p
= \|\varphi_x\|$, where $\|\cdot\|_p$ denotes the norm on $L^p_{(0)}(\varphi)$.
\end{proposition}
\begin{proof}
Firstly, we show that
\[ \|j^*_{(0)}(x)\|_p \leq \|x\|^{1/p'} \|\varphi_x\|^{1/p} \qquad (x\in L). \]
This is \cite[Corollary~2.8]{staf}, but we give a quick proof.  Pick $\epsilon>0$ and
define $F:\mc S\rightarrow L$ by $F(z) = \exp(\epsilon(z^2-\theta^2)) \|\varphi_x\|^{\theta-z}
\|x\|^{z-\theta} x$.  Then $F\in\mc F$, $F(\theta) = x$, and we can check that
\[ \|F\|_{\mc F} \leq \|\varphi_x\|^\theta \|x\|^{1-\theta} \exp\big( \epsilon(1-\theta^2) \big). \]
As $\epsilon>0$ was arbitrary, we conclude that, as $\theta=1/p$,
\[ \|j^*_{(0)}(x)\|_p \leq \|x\|^{1-\theta} \|\varphi_x\|^{\theta} =
\|x\|^{1/p'} \|\varphi_x\|^{1/p}. \]

We now use duality.  For $\epsilon>0$, there exists $p_0>1$ such that, if $1<p\leq p_0$,
then $\|j^*_{(0)}(x)\|_p \leq (1+\epsilon)\|\varphi_x\|$.  As $L$ is $\sigma$-weakly dense in
$\mc M$, by Kaplansky density, we can find $y\in L$ with $\|y\|=1$ and $|\ip{y}{\varphi_x}|
\geq (1-\epsilon)\|\varphi_x\|$.  Then there exists $p_1>1$ such that, if $1<p\leq p_1$,
then $\|j^*_{(0)}(y)\|_{p'} \leq (1+\epsilon) \|y\| = 1+\epsilon$.  Thus, if $1<p<\min(p_0,p_1)$,
then
\begin{align*} (1+\epsilon) \|\varphi_x\| &\geq \|j^*_{(0)}(x)\|_p
\geq \big| \ip{j^*_{(0)}(x)}{j^*_{(0)}(y)}_{p,(0)} \big|
\|j^*_{(0)}(y)\|_{p'}^{-1} \\
&\geq |\ip{y}{\varphi_x}| (1+\epsilon)^{-1}
\geq (1-\epsilon) (1+\epsilon)^{-1} \|\varphi_x\|. \end{align*}
As $\epsilon>0$ was arbitrary, this completes the proof.
\end{proof}

\subsection{Operator spaces}

As noted by Pisier in \cite{pisier_mem}, \cite[Section~2.6]{pisier}, the complex interpolation
method interacts very nicely with operator spaces.  If $E_0$ and $E_1$ are operator spaces
which, as Banach spaces, form a compatible couple, then, say, identifying
$\mathbb M_n(E_0+F_0)$ with $(E_0+F_0)^{n^2}$, we turn $(\mathbb M_n(E_0),
\mathbb M_n(E_1))$ into a compatible couple.  We then define
\[ \mathbb M_n\big( (E_0,E_1)_{[\theta]} \big) =
(\mathbb M_n(E_0), \mathbb M_n(E_1))_{[\theta]}. \]
It is an easy check that these matrix norms satisfy the axioms for an (abstract)
operator space.  Then the obvious completely bounded version of Theorem~\ref{thm:inter}
holds.

Suppose that $E$ and $F$ are Banach spaces which form a sesquilinear dual pair.  A typical
example would be $E=L^\infty(\mu)$ and $F=L^1(\mu)$ for a probability measure $\mu$, together
with the pairing
\[ (f|g) = \int f \overline{g} \ d\mu \qquad (f\in E, g\in F). \]
Then we can show that $(E,F)_{[1/2]}$ is a Hilbert space, if $(E,F)$ is made a compatible
couple in the correct way, see \cite[Theorem~7.10]{pisier} for example.  In our example,
we recover $L^2(\mu)$ for the canonical compatibility.
Intrinsic in the proof is that a Hilbert space $H$ can be canonically identified, in an
anti-linear way, with its own dual, by way of the inner-product.

If $E$ and $F$ are also operator spaces, then we recover a Hilbert space with some operator
space structure.  There is a unique operator Hilbert space which is anti-linearly completely
isometric to its dual: Pisier's \emph{operator Hilbert space}.  We write $H_{oh}$ to denote
this structure on $H$.  As explained carefully in \cite[Page~139]{pisier}, at least when $\mc M$
is semifinite, we should consider the compatible couple $(\mc M,\mc M_*^\op)$.  Here, for
an operator space $E$, $E^\op$ denotes the space $E$ with the \emph{opposite} structure,
namely $\|(x_{ij})\|_\op = \|(x_{ji})\|$ for $(x_{ij})\in\mathbb M_n(E)$.  If $\mc A$ is a
C$^*$-algebra, then $\mc A^\op$ can be identified with $\mc A$, but with the product
reversed.  See also \cite[Section~4]{jrx} for a slightly different perspective.

Indeed, as noted in \cite{jrx}, if $\mc M$ is in standard position on $H$ with modular
conjugation $J$, then we have a canonical $*$-isomorphism $\phi:\mc M^\op \rightarrow \mc M',
x \mapsto Jx^*J$ and so $\phi_*:\mc M'_* \rightarrow \mc M_*^\op$ is a completely isometric
isomorphism of operator spaces.  We conclude that the natural operator space structure
on $L^p(\mc M)$ will arise from studying the compatible couple $(\mc M,\mc M'_*)$.
Alternatively, if we privilege $\mc M_*$, then we should look at $(\mc M',\mc M_*)$.
When $\mc M_*=A(G)$, it turns out that this simple observation will guide us as to how to
give the resulting non-commutative $L^p$ spaces an $A(G)$-module action.

Let us finish by showing the operator space version of Proposition~\ref{prop:limit}.

\begin{proposition}\label{prop:limitcb}
With notation as above, let $x\in \mathbb M_n(L)$ for some $n\in\mathbb N$.
Then $\lim_{p\rightarrow 1} \|j^*_{(0)}(x)\|_p = \|\varphi_x\|$.
\end{proposition}
\begin{proof}
The norm on $\mathbb M_n(L^p_{(0)}(\varphi))$ is given by interpolating
$\mathbb M_n(\mc M)$ and $\mathbb M_n(\mc M_*^\op)$, and so we can follow the first
part of the proof of Proposition~\ref{prop:limit} to find $p_0>1$ such that, if
$1<p<p_0$, then $\|j^*_{(0)}(x)\|_p \leq (1+\epsilon) \|\varphi_x\|$.

By Smith's Lemma, \cite[Proposition~2.2.2]{ER}, and as $\mathbb M_n(L)$ is $\sigma$-weakly 
dense in $\mathbb M_n(\mc M)$, there exists $y\in\mathbb M_n(L)$ with $\|y\|=1$ and
$|\langle\ip{y}{\varphi_x}\rangle|\geq (1-\epsilon)\|\varphi_x\|$.  We can now proceed as
in the end of the proof of Proposition~\ref{prop:limit}.
\end{proof}

\section{Non-commutative $L^p$ spaces associated to the Fourier algebra}
\label{sec:third}

Let $G$ be a locally compact group $G$.  We have that $VN(G)$ is a Hopf-von Neumann
algebra; indeed, a Kac algebra, \cite{ES}; indeed, is a locally compact quantum
group, \cite{kusvn, kusbook}.  We have a normal $*$-homomorphism
\[ \Delta:VN(G)\rightarrow VN(G)\vnten VN(G)=VN(G\times G);
\quad \Delta(\lambda(s)) = \lambda(s)\otimes\lambda(s) \qquad (s\in G). \]
It is not obvious that such a map exists, but if we define $W\in\mc B(L^2(G\times G))$
by $W\xi(s,t) = \xi(ts,t)$ for $s,t\in G$, then $W$ is a unitary, and we can define
$\Delta(x) = W^*(1\otimes x)W$ for $x\in VN(G)$.  Then $\Delta$ is coassociative, namely
$(\Delta\otimes\id)\Delta = (\id\otimes\Delta)\Delta$.  Thus $\Delta$ induces an
associative product on $VN(G)_*$, leading to the Fourier algebra, \cite{eymard}.
For $\xi,\eta\in L^2(G)$ we write $\omega_{\xi,\eta}$ for the normal functional on
$VN(G)$ given by $\ip{x}{\omega_{\xi,\eta}} = (x\xi|\eta)$.  As $VN(G)$ is in
standard position, \cite[Chapter~IX, Section~1]{tak2}, every member of $A(G)$ arises in this way.
We define a map, the \emph{Eymard embedding}, $\Phi:VN(G)_* \rightarrow C_0(G)$ by
\[ \Phi(\omega_{\xi,\eta})(s) = \ip{\lambda(s)}{\omega_{\xi,\eta}}
= \int_G \xi(s^{-1}t) \overline{\eta(t)} \ dt
\qquad (s\in G, \omega_{\xi,\eta}\in A(G)). \]
Then $\Phi$ is an algebra homomorphism.  This follows \cite{eymard}, but we warn
the reader that \cite[Chapter~VII, Section~3]{tak2} uses a different map (with $s^{-1}$ replacing $s$).

Then $VN_r(G) = VN(G)'$ carries a coassociative map $\Delta'$ given by
$\Delta'(x) = (J\otimes J)\Delta(JxJ)(J\otimes J)$ for $x\in VN_r(G)$.  We have that
$\Delta'(\rho(s)) = \rho(s) \otimes \rho(s)$ for $s\in G$.  Similarly $A_r(G) = VN_r(G)_*$
becomes an algebra.  We write $\omega'_{\xi,\eta}$ for the functional on $VN_r(G)$ given by $\ip{x}{\omega'_{\xi,\eta}} = (x\xi|\eta)$ for $x\in VN_r(G)$.  We similarly define $\Phi':
A_r(G)\rightarrow C_0(G)$ by
\[ \Phi'(\omega'_{\xi,\eta})(s) = \ip{\rho(s)}{\omega_{\xi,\eta}}
= \int_G \xi(ts)\nabla(s)^{1/2} \overline{\eta(t)} \ dt
\qquad (s\in G, \omega'_{\xi,\eta}\in A_r(G)). \]

Guided by the arguments in the previous section, we shall turn $(VN_r(G),A_r(G))$ into
a compatible couple in the sense of Terp.  As $VN_r(G) = VN(G)'$, we have a canonical
$*$-isomorphism
\[ \phi:VN(G)^\op \rightarrow VN_r(G); \quad x\mapsto Jx^*J
\qquad (x\in VN(G)). \]
Then we have
\[ \phi_*:A_r(G) \rightarrow A(G)^\op; \quad \omega'_{\xi,\eta} \mapsto \omega_{J\eta,J\xi}
\qquad (\xi,\eta\in L^2(G)). \]
This allows us to regard $(VN_r(G),A(G))$ as a compatible couple, and we shall often
suppress the implicit $\phi_*$ involved.  We then define
\[ L^p(\hat G) = (VN_r(G),A(G))_{[1/p]} \qquad (1<p<\infty). \]
Here we use the ``dual group'' notation which is common when studying the Fourier algebra.
The motivation is that when $G$ is abelian, we have that $VN_r(G) = L^\infty(\hat G)$ and
$A(G)=L^1(\hat G)$ by the Fourier transform, where $\hat G$ is the Pontryagin dual of
$G$, and so $L^p(\hat G)$ agrees with the usual meaning.  We keep the same notation in the
non-abelian case, although now it is purely formal.  We give
$L^p(\hat G)$ the canonical operator space structure
\[ \mathbb M_n(L^p(\hat G)) = \big( \mathbb M_n(VN_r(G)), \mathbb M_n(A(G)) \big)_{[1/p]}. \]

We have that
\[ \Phi(\phi_*(\omega'_{\xi,\eta}))(s) = (J\lambda(s)^*J\xi|\eta)
= (\rho(s^{-1})\xi|\eta) = \Phi'(\omega'_{\xi,\eta})(s^{-1})
\qquad (s\in G, \xi,\eta\in L^2(G)). \]
Hence, under the maps $\Phi$ and $\Phi'$, $\phi_*$ induces the ``check map''.
We also have the map $K$ available, which allows us to define a $*$-homomorphism
\[ \hat\phi:VN(G)\rightarrow VN_r(G); \quad x\mapsto KxK \qquad (x\in VN(G)). \]
The predual of this map is then
\[ \hat\phi_*:A_r(G)\rightarrow A(G); \quad \omega'_{\xi,\eta} \mapsto \omega_{K\xi,K\eta}
\qquad (\xi,\eta\in L^2(G)), \]
so that
\[ \Phi(\hat\phi_*(\omega'_{\xi,\eta}))(s) = (K\lambda(s)K\xi|\eta)
= (\rho(s)\xi|\eta) = \Phi'(\omega'_{\xi,\eta})(s)
\qquad (s\in G, \xi,\eta\in L^2(G)). \]
Thus, under the maps $\Phi$ and $\Phi'$, we see that $\hat\phi_*$ is the formal identity.

\begin{lemma}\label{lem:int_rrr}
For $f,g\in C_{00}(G)$, let $a=f^*g$.  Then $\rho(a)\in VN_r(G)$ agrees with $\nabla^{1/2} a
\in A(G)$ in $VN_r(G)\cap A(G)=L$.
\end{lemma}
\begin{proof}
We have that $\rho(f),\rho(g)\in\mf n_{\varphi'}$, and so by
\cite[Proposition~4]{terp}, we have that $\rho(f^*g) \in L = VN_r(G)\cap A(G)$, with
\[ \varphi_{\rho(f^*g)} = \omega'_{J\Lambda'\rho(f),J\Lambda'\rho(g)}
= \omega'_{\overline{f},\overline{g}} = \phi_*^{-1}(\omega_{\Lambda'\rho(g),\Lambda'\rho(f)})
= \phi_*^{-1}(\omega_{Kg,Kf}) = \phi_*^{-1} \hat\phi_*(\omega'_{g,f}). \]
Now, for $s\in G$,
\begin{align*} \omega_{Kg,Kf}(s) &= \int_G Kg(s^{-1}t) \overline{Kf(t)} \ dt
= \int_G g(t^{-1}s) \nabla(t^{-1}s)^{1/2} \overline{f(t^{-1})} \nabla(t^{-1})^{1/2}
\ dt \\ &= \nabla(s)^{1/2} \int_G f^*(t) g(t^{-1}s) \ dt
= (\nabla^{1/2} a)(s), \end{align*}
which completes the proof.
\end{proof}

We wish to turn $L^p(\hat G)$ into a (completely contractive) left $A(G)$-module.
For $p=1$, we obviously have a natural action of $A(G)$ on itself, and so the
previous lemma suggests the following action.

\begin{lemma}\label{lem:mod_struc}
There is a completely contractive action of $A(G)$ on $VN_r(G)$ such that
$a\cdot \rho(f) = \rho(a\cdot f)$ for $a\in A(G)$ and $f\in C_{00}(G)$, where $a\cdot f$
denotes the point-wise product.
\end{lemma}
\begin{proof}
We have that $VN_r(G)$ is a completely contractive $A_r(G)$-module (which is commutative,
so we shall not distinguish between left and right actions) such that $a\cdot\rho(f)
=\rho(a\cdot f)$ for $a\in A_r(G)$ and $f\in C_{00}(G)$.  As above, we have that
$\hat\phi_*:A_r(G)\rightarrow A(G)$ is a completely isometric homomorphism.  So our
required action is simply $a\cdot x = \hat\phi_*^{-1}(a)\cdot x$ for $a\in A(G), x\in
VN_r(G)$.
\end{proof}

The following is a useful approximation result, which allows us to work with concrete
functions, rather than operators in $VN_r(G)$.

\begin{proposition}\label{prop:den_one}
For $x\in VN_r(G)$, we have that $x\in L$ when there exists $\varphi_x\in A_r(G)$ with
$(x(\overline{a})|\overline b) = \ip{\rho(a^*b)}{\varphi_x}$ for $a,b\in C_{00}(G)$.
\end{proposition}
\begin{proof}
Let $\mf A = \Lambda'(\rho(C_{00}(G))) = C_{00}(G)$, which is a Tomita algebra (but
\emph{not} the maximal Tomita algebra).  We claim that $\mf A$ generates the full left
Hilbert algebra $\Lambda'(\mf n_{\varphi'} \cap \mf n_{\varphi'})$.  This will follow
from \cite[Lemma~3, Section~10.5]{sz} if we can show that $C_{00}(G)$ is a core for the
operator $S$, which is the closure of $\Lambda'(x)\mapsto\Lambda'(x^*)$ for
$x\in\mf n_{\varphi'} \cap \mf n_{\varphi'}$ (meaning that the closure of the $S$ operator
associated to $\mf A$ agrees with the canonical one associated to $\mf n_{\varphi'}$).

Indeed, for us, $S$ is the map $D(S)\rightarrow L^2(G), \xi\mapsto \overline{\check\xi}$
where $D(S) = \{ \xi\in L^2(G) : \check\xi\in L^2(G) \}$.  Then $D(S)$ is a Hilbert
space for the inner-product $(\xi|\eta)_\sharp = (\xi|\eta) + (S\eta|S\xi)$ for $\xi,\eta
\in D(S)$.  We claim that $C_{00}(G)$ is dense in $D(S)$, from which it will follow
that $C_{00}(G)$ is a core for $S$.  Suppose that $\eta\in D(S)$ is such that
$(\xi|\eta)_\sharp=0$ for $\xi\in C_{00}(G)$.  Then
\[ 0 = \int_G \xi(s) \overline{\eta(s)} \ ds + \int_G \xi(s^{-1}) \overline{\eta(s^{-1})} \ ds
= \int_G \overline{\eta(s)} \xi(s) (1+\nabla(s)^{-1}) \ ds, \]
for all $\xi\in C_{00}(G)$.  As the set $\{ \xi\cdot(1+\nabla^{-1}) : \xi\in C_{00}(G) \}$
is dense in $L^2(G)$, it follows that $\eta=0$.  So $C_{00}(G)$ is dense in $D(S)$, as required.

As $\mf A$ generates $\Lambda'(\mf n_{\varphi'} \cap \mf n_{\varphi'})$, we can
apply the approximation result \cite[Theorem~1.26, Chapter~VI]{tak2}.  This shows that for
$x\in \mf n_{\varphi'}$, we can find a sequence $(f_n)$ in $C_{00}(G)$ such that
\[ \lim_n \|\Lambda'(x) - \Lambda'\rho(f_n)\|=\lim_n \|\Lambda'(x) - Kf_n\| = 0, \qquad
\|\rho(f_n)\| \leq \|x\| \qquad (n\in\mathbb N), \]
and that $\rho(f_n)\rightarrow x$ strongly.

Finally, suppose that $x\in VN_r(G)$ and $\varphi_x\in A_r(G)$ are such that
$(x(\overline a)|\overline b) = \ip{\rho(a^*b)}{\varphi_x}$ for $a,b\in C_{00}(G)$.
Choose $\xi,\eta\in L^2(G)$ with $\varphi_x = \omega'_{\xi,\eta}$.
Let $y,z\in\mf n_{\varphi'}$, so we can find sequences $(a_n),(b_n)$ in $C_{00}(G)$,
as above, associated to $y$ and $z$ respectively.  Thus
\begin{align*} \ip{y^*z}{\varphi_x} &= (z\xi|y\eta) = \lim_n (\rho(b_n)\xi|\rho(a_n)\eta)
= \lim_n \ip{\rho(a_n^*b_n)}{\varphi_x} \\ &= \lim_n (xJKa_n|JKb_n)
= (xJ\Lambda'(y)|J\Lambda'(z)). \end{align*}
We conclude that $x\in L$ as required.
\end{proof}

We can immediately improve Lemma~\ref{lem:int_rrr}.

\begin{proposition}\label{prop:int_rrr_better}
Let $a\in A(G)$.  Then $a\in VN_r(G)\cap A(G)$ if and only if $\check a$ is \emph{right
bounded}, that is, there exists $K>0$ such that $\|f\check a\|_2 \leq K\|f\|_2$ for
$f\in C_{00}(G)$.  In this case, the map $f\mapsto f \check a$ extends to an operator
$x\in VN_r(G)$, and then $x\in L$ with $a = \phi_*(\varphi_x)$.
\end{proposition}
\begin{proof}
Let $a=\omega_{\xi,\eta}$, so that $\check a = \Phi' \phi_*^{-1}(\omega_{\xi,\eta})$.  
Suppose that $\check a$ is right bounded.  As convolutions on the right commutes with
the action of $VN(G)$, we see that $x\in VN_r(G)$.  For $f,g\in C_{00}(G)$, we see that
\begin{align*} (x(\overline f)|\overline g) &= (\overline f \check a|\overline g)
= \int \overline{f(s)} \check a(s^{-1}t) g(t) \ ds \ dt
= \int \overline{f(s)} \check a(t) g(st) \ dt \ ds \\
&= \int f^*(s) g(s^{-1}t) \check a(t)  \ ds \ dt
= \ip{\rho(f^*g)}{\phi_*^{-1}(\omega_{\xi,\eta})}. \end{align*}
So by the previous proposition, $x\in L$ and $a = \phi_*(\varphi_x)$, as claimed.

Conversely, if $a\in VN_r(G)\cap A(G)$ then there exists $x\in L$ with
$a = \phi_*(\varphi_x)$.  As $f\check a$ always exists for $f\in C_{00}(G)$,
we can reverse the argument above to conclude that $x(f) = f\check a$ for
$f\in C_{00}(G)$, so that $\check a$ is right bounded.
\end{proof}

We can also apply our approximation idea to improve an approximation
result of Terp, \cite[Theorem~8]{terp}.

\begin{proposition}\label{prop:den_two}
For $x\in L$, we can find a net $(f_i)$ in $C_{00}(G)^2$ such that
$\sup_i \|\rho(f_i)\| < \infty$, $\rho(f_i)\rightarrow x$ $\sigma$-weakly, and
$\varphi_{\rho(f_i)} \rightarrow \varphi_x$ in norm.
\end{proposition}
\begin{proof}
By Terp's result \cite[Theorem~8]{terp} we can a net bounded $(x_i)$ in $\mf m_{\varphi'}$
with $x_i\rightarrow x$ $\sigma$-weakly and $\varphi_{x_i}\rightarrow\varphi_x$ in norm.
Indeed, from the proof, we can choose $x_i = y_i^*z_i$ for some $y_i,z_i\in
\mf n_{\varphi'}$ with $(y_i)$ and $(z_i)$ bounded nets.

For each $i$, choose a sequence $(a_{i,n})$ in $C_{00}(G)$ with
$\rho(a_{i,n}) \rightarrow y_i$ strongly, $Ka_{i,n}\rightarrow\Lambda'(y_i)$
in norm, and with $\|\rho(a_{i,n})\|\leq\|y_i\|$.  Similarly choose $(b_{i,n})$ associated
to $z_i$.  It follows (compare with the proof above) that $\rho((a_{i,n})^* b_{i,n})
\rightarrow y_i^*z_i = x_i$ $\sigma$-weakly, and that $\varphi_{\rho((a_{i,n})^* b_{i,n})}
\rightarrow \varphi_{x_i}$ in norm.  With the diagonal ordering, we see that
$((a_{i,n})^* b_{i,n})$ is the required net.
\end{proof}

\begin{theorem}\label{thm:ag_mod_action}
There is a completely contractive left action of $A(G)$ on $L^p(\hat G)$, for $1<p<\infty$,
such that $a\cdot j^*_{(0)}\rho(b) = j^*_{(0)}\rho(a\cdot b)$ for $a\in A(G)$ and
$b\in C_{00}(G)^2$.
\end{theorem}
\begin{proof}
Let $a\in A(G)$ and consider the bounded maps
\[ T:A(G)\rightarrow A(G); b\mapsto a\cdot b, \qquad
S:VN_r(G)\rightarrow VN_r(G); x\mapsto \hat\phi_*^{-1}(a)\cdot x. \]
By Lemma~\ref{lem:when_compat}, we
wish to show that $T$ and $S$ map $L$ to $L$ and agree on $L$.  If this is so,
then we get a map $R\in\mc B(L^p(\hat G))$ which extends $T$ and $S$,
and is bounded by $\|T\|^{1/p} \|S\|^{1/p'} \leq \|a\|$.  Clearly $a\mapsto R$
is a homomorphism, and the resulting action of $A(G)$ on $L^p(\hat G)$ is the one
stated, by Lemma~\ref{lem:mod_struc}.

So, for $x\in L$, we need to show that $y=\hat\phi_*^{-1}(a)\cdot x\in L$
and that furthermore $a\cdot\phi_*(\varphi_x) = \phi_*(\varphi_y)$.
Suppose that $x=\rho(f^*g)$ for $f,g\in C_{00}(G)$, so that from Lemma~\ref{lem:int_rrr},
$\Phi\phi_*(\varphi_x) = \nabla^{1/2} f^*g$.  By Proposition~\ref{prop:den_one},
we have that $y\in L$ if
\[ (y(\overline{c})|\overline{d}) = \ip{\rho(c^*d)}{\phi_*^{-1}(a\cdot\phi_*(\varphi_x))}
\qquad (c,d\in C_{00}(G)). \]
Now, we have that
\begin{align*}
& \ip{\rho(c^*d)}{\phi_*^{-1}(a\cdot\phi_*(\varphi_x))}
= \ip{\rho(c^*d)}{\phi_*^{-1}\Phi^{-1}(a\cdot \nabla^{1/2} f^*g)} \\
&= \int_G c^*d(s) a(s^{-1}) \nabla(s)^{-1/2} (f^*g)(s^{-1}) \ ds 
= \int_G c^*d(s^{-1}) a(s) \nabla(s)^{-1/2} (f^*g)(s) \ ds \\
&= \ip{\hat\phi_*^{-1}(a) \cdot \rho(f^*g)}{\varphi_{\rho(c^*d)}}
= \ip{y}{\omega'_{\overline{c},\overline{d}}} = (y(\overline{c})|\overline{d}),
\end{align*}
using that $\Phi'\phi_*^{-1}\Phi^{-1}$ is the check map.  Hence we are done in the
case that $x\in\rho(C_{00}(G)^2)$.

For general $x\in L$, choose an approximating net $(f_i)\subseteq C_{00}(G)^2$
as in Proposition~\ref{prop:den_two}.  Then, by the previous paragraph,
for $a,b\in C_{00}(G)$,
\[ \ip{y}{\omega'_{\overline{c},\overline{d}}} =
\lim_i \ip{\rho(f_i)}{\omega'_{\overline{c},\overline{d}}}
= \lim_i \ip{\rho(c^*d)}{\phi_*^{-1}(a\cdot\phi_*(\varphi_{\rho(f_i)}))}
= \ip{\rho(c^*d)}{\phi_*^{-1}(a\cdot\phi_*(\varphi_x))}, \]
which completes the proof of the claim, by another application of Lemma~\ref{lem:int_rrr}.

In the completely bounded case, notice that $T$ and $S$ are completely bounded,
and hence also $R$ is, so we get a homomorphism $A(G)\rightarrow\mc{CB}(L^p(\hat G))$.
To see that this is completely bounded, it is easier to prove the equivalent statement
that $A(G)\times L^p(\hat G) \rightarrow L^p(\hat G); (a,\xi)\mapsto R(\xi)$ is jointly
completely contractive, \cite[Chapter~7]{ER}.  That is, for $n\in\mathbb N$, the map
$\mathbb M_n(A(G)) \times \mathbb M_n(L^p(\hat G)) \rightarrow
\mathbb M_{n^2}(L^p(\hat G)); (a_{ij})\times(\xi_{kl}) \mapsto
(R_{ij}(\xi_{kl}))_{(i,k),(j,l)}$ is contractive.  This follows immediately from
Theorem~\ref{thm:bilin_int}, as the analogous statements hold for $T$ and $S$.
\end{proof}

A slightly curious corollary of this proof is that $L$ is an $A(G)$-submodule
of $VN_r(G)$, and hence the image of $L$ in $A(G)$ is a dense ideal.
As a final application of our approximation ideas, we have the following.

\begin{proposition}\label{prop:den_in_lp}
For $1<p<\infty$, we have that $j_{(0)}^* \rho(C_{00}(G)^2)$ is norm dense in
$L^p(\hat G)$.
\end{proposition}
\begin{proof}
Following the proof of \cite[Proposition~6.22]{izumi2}, it suffices to show that
$\rho(C_{00}(G)^2) \subseteq L$ separates the points of $VN_r(G)+A_r(G) \subseteq L^*$.
Indeed, suppose that $x\in VN_r(G)$ and $\omega\in A_r(G)$ are such that
\[ \ip{x}{\varphi_{\rho(f)}} + \ip{\rho(f)}{\omega} = 0 \qquad (f\in C_{00}(G)^2). \]
For $y\in L$, use Proposition~\ref{prop:den_two} to pick an approximating net $(f_i)$,
so that
\[ 0 = \lim_i \ip{x}{\varphi_{\rho(f_i)}} + \ip{\rho(f_i)}{\omega}
= \ip{x}{\varphi_y} + \ip{y}{\omega}. \]
In particular, this holds for any $y\in\mf m_{\varphi}$, so by \cite[Proposition~7]{terp}
(or, essentially by definition) it follows that $x\in L$ with $\varphi_x = -\omega$.
Hence $x+\omega=0$ in $VN_r(G)+A_r(G)$, as required.
\end{proof}

\begin{proposition}\label{prop:std_hilbert}
There is an isometric isomorphism $\theta:L^2(G) \rightarrow L^2(\hat G)$
satisfying $\theta(f) = j^*_{(0)}\rho(\nabla^{-3/4}\check f)$ for $f\in C_{00}(G)^2$.
Furthermore, $\theta$ intertwines the inner products on $L^2(G)$ and $L^2(\hat G)$.
\end{proposition}
\begin{proof}
In Section~\ref{sec:second}, we discussed the isometric isomorphism
$k:H_{\varphi'}\rightarrow L^2(\hat G)$ which is such that
$k(\Lambda'\rho(f)) = j^*_{(0)}\sigma_{i/4}\rho(f)$ for $f\in C_{00}(G)^2$.
If we identify $H_{\varphi'}$ with $L^2(G)$, then $\Lambda'\rho(f) = Kf$, and so we
find a map $\theta$ which satisfies $\theta(f) = j^*_{(0)}\sigma_{i/4}\rho(Kf)
= j^*_{(0)}\rho(\nabla^{-3/4}\check f)$ for $f\in C_{00}(G)^2$.
\end{proof}

Notice that $L^2(G)$ carries a natural \emph{bilinear} product, $\ip{f}{g} = \int_G f g$ for
$f,g\in L^2(G)$.  Similarly $L^2(\hat G)$ has the bilinear product $\ip{\cdot}{\cdot}_{2,(0)}$,
but $\theta$ does not intertwine these products.

\subsection{Comparison with Forrest, Lee and Samei}

In \cite[Section~6]{fls}, a different construction of non-commutative $L^p$ spaces associated
to $A(G)$ is given.  We shall compare their construction to ours.

Firstly, they form the non-commutative $L^p$ space using $VN(G)$, using
Izumi's work with $\alpha=-1/2$.  Let $\mc OL^p_{(-1/2)}(VN(G))$ be the operator space
version, given by interpolating $VN(G)$ and $A(G)^\op$.  Here we write $(-1/2)$ to
indicate the choice of $\alpha$.  Then they define
\[ L^p(VN(G)) = \begin{cases} \mc O L^p_{(-1/2)}(VN(G))^\op &: 1< p\leq 2, \\
\mc O L^p_{(-1/2)}(VN(G)) &: 2\leq p<\infty. \end{cases} \]
Recall that the Hilbert space $\mc OL^2_{(-1/2)}(VN(G))$ will carry the operator Hilbert space
structure, so that $\mc OL^2_{(-1/2)}(VN(G)) = \mc OL^2_{(-1/2)}(VN(G))^\op$.

By \cite[Theorem~6.3]{fls}, for $1<p\leq 2$, the $A(G)$ module structure on $L^p(VN(G))$
satisfies
\[ a \cdot j^*_{(1/2)}(\lambda(\check f)) = j^*_{(1/2)}(\lambda(\check a \cdot \check f))
\qquad (a\in A(G), f\in C_{00}(G)). \]
Here we use a different, but equivalent, notation to that of \cite{fls}.
Similarly, by \cite[Theorem~6.4]{fls}, for $2\leq p<\infty$, the module action of $A(G)$ on
$L^p(VN(G))$ is
\[ a\cdot j^*_{(1/2)}(\lambda(f)) = j^*_{(1/2)}(\lambda(a\cdot f))
\qquad (a\in A(G), f\in C_{00}(G)). \]

Recall the isometric isomorphism $U_{p,(0,-1/2)}:L^p_{(-1/2)}(VN(G))
\rightarrow L^p_{(0)}(VN(G))$ which satisfies, in particular,
\[ U_{p,(0,-1/2)}( j^*_{(1/2)}\lambda(f) ) = j^*_{(0)}(\sigma_{i/2p}\lambda(f))
= j^*_{(0)} \lambda(\Delta^{-1/2p} f) \qquad (f\in C_{00}(G)^2). \]
For $1<p\leq 2$, we can hence regard $L^p(VN(G))$ as $\mc OL^p_{(0)}(VN(G))^\op$
with the module action
\begin{align*} a \cdot j^*_{(0)}\lambda(f)
&= U_{p,(0,-1/2)}( a\cdot U_{p,(0,-1/2)}^{-1}j^*_{(0)}\lambda(f))
=  U_{p,(0,-1/2)}( a\cdot j^*_{(0)}\lambda(\Delta^{1/2p}f)) \\
&= U_{p,(0,-1/2)} j^*_{(0)}\lambda(\check a \cdot \Delta^{1/2p}f)
= j^*_{(0)}\lambda(\check a \cdot f) . \end{align*}
for $a\in A(G)$ and $f\in C_{00}(G)^2$.  Similarly, for $2\leq p<\infty$, we
regard $L^p(VN(G))$ as $\mc OL^p_{(0)}(VN(G))$ with the module action
\[ a\cdot j^*_{(0)}\lambda(f) = j^*_{(0)}\lambda(a\cdot f)
\qquad (a\in A(G), f\in C_{00}(G)^2). \]

\begin{proposition}
For $2\leq p<\infty$, there exists a completely isometric isomorphism
$\hat\phi_p: L^p(VN(G)) \rightarrow L^p(\hat G)$
which is also an $A(G)$-module homomorphism, with
\[ \hat\phi_p\big( j^*_{(0)} \lambda(f)\big) = j^*_{(0)}\rho(f) \qquad (f\in C_{00}(G)^2). \]
\end{proposition}
\begin{proof}
Our $L^p(\hat G)$ spaces are formed by interpolating $VN_r(G)$ and $A_r(G)^\op$
(identified with $A(G)$).
Consider again the maps $\hat\phi:VN(G) \rightarrow VN_r(G)$ and
$\hat\phi_*^{-1}:A(G)\rightarrow A_r(G)$.  We claim that these are compatible, that is,
map $L_{(0)}$, for $(VN(G),A(G))$, into $L_{(0)}$, for $(VN_r(G),A_r(G))$.  Indeed,
let $x\in L_{(0)} \subseteq VN(G)$ with associated $\varphi_x\in A(G)$.  Let $a,b\in C_{00}(G)$,
so that
\begin{align*} (\hat\phi(x) \overline a| \overline b) &=
(x K(\overline a)| K(\overline b)) = (xJ(a)|J(b))
= (xJ\Lambda(\lambda(a))| J\Lambda(\lambda(b)) \\
&= \ip{\lambda(a^*b)}{\varphi_x}
= \ip{K\rho(a^*b)K}{\varphi_x}
= \ip{\rho(a^*b)}{\hat\phi_*^{-1}(\varphi_x)}, \end{align*}
so by Proposition~\ref{prop:den_one} we see that $\hat\phi(x)\in L_{(0)}$, for $(VN_r(G),A_r(G)$,
with $\varphi_{\hat\phi(x)} = \hat\phi_*^{-1}(\varphi_x)$.  Consequently, by
Lemma~\ref{lem:when_compat}, we can
interpolate these maps, leading to a contraction
\[ \hat\phi_p : L^p_{(0)}(VN(G)) \rightarrow L^p_{(0)}(VN_r(G)) = L^p(\hat G). \]
As $\hat\phi_*^{-1}$ is also a complete isometry $A(G)^\op\rightarrow A_r(G)^\op$,
we see that $\hat\phi_p$ is even a complete contraction.
By symmetry, we also have a complete contraction in the other direction, showing that
$\hat\phi_p$ is actually a completely isometric isomorphism.  In particular,
\[ \hat\phi_p j^*_{(0)} \lambda(f) = j^*_{(0)} \big( K\lambda(f)K \big)
= j^*_{(0)} \rho(f) \qquad (f\in C_{00}(G)^2). \]
It is now clear from Proposition~\ref{prop:den_in_lp} that this map is an $A(G)$-module
homomorphism.
\end{proof}

\begin{proposition}
For $1< p\leq 2$, there exists a completely isometric isomorphism
$\phi_p: L^p(VN(G)) \rightarrow L^p(\hat G)$
which is also an $A(G)$-module homomorphism, with
\[ \phi_p\big( j^*_{(0)} \lambda(f)\big) = j^*_{(0)}\rho(\check f\nabla^{-1})
\qquad (f\in C_{00}(G)^2). \]
\end{proposition}
\begin{proof}
For $1<p\leq 2$, it is clear that $L^p(VN(G)) = \mc OL^p_{(0)}(VN(G))^\op =
( VN(G)^\op, A(G) )_{[1/p]}$.  The idea now is to replicate the proof above,
but using instead the maps $\phi:VN(G)^\op \rightarrow VN_r(G)$ and
$\phi_*^{-1}:A(G) \rightarrow A_r(G)^\op$.  For $x\in L \subseteq VN(G)$, let
$\varphi_x = \omega_{\xi,\eta}$ for some $\xi,\eta\in L^2(G)$.  Let $y',z'\in
\mf n_{\varphi'}$ so that $y=Jy'J, z=Jz'J \in \mf n_{\varphi}$ and
\begin{align*} (\phi(x) J\Lambda'(y') | J\Lambda'(z'))
&= (Jx^*J \Lambda(y)|\Lambda(z)) = \ip{z^*y}{\varphi_x}
= (z^*y\xi|\eta) = (J (z')^* y' J\xi|\eta) \\
&= ((y')^*z'J\eta|J\xi) = \ip{(y')^*z'}{\phi_*^{-1}(\varphi_x)}. \end{align*}
Hence $\phi(x)\in L\subseteq VN_r(G)$ with $\varphi_{\phi(x)} = \phi_*^{-1}(\varphi_x)$.
Again, we interpolate to find a completely isometric isomorphism
\[ \phi_p:L^p(VN(G)) \rightarrow L^p(\hat G). \]
We then see that for $f\in C_{00}(G)^2$,
\[ \phi_p j^*_{(0)} \lambda(f) = j^*_{(0)} \phi(\lambda(f))
= j^*_{(0)}\big( J \lambda(f^*) J\big)
= j^*_{(0)} \rho( \check f \nabla^{-1} ). \]
It is now clear from Proposition~\ref{prop:den_in_lp} that this map is an $A(G)$-module
homomorphism.
\end{proof}

\subsection{Application to homological questions}

The following is an improvement of \cite[Proposition~6.8]{fls}, which
only showed the result for $p\geq 2$.

\begin{proposition}
Let $G$ be a non-discrete group, and let $1<p<\infty$.  Then the only bounded
left $A(G)$-module homomorphism $L^p(\hat G) \rightarrow A(G)$ is the zero map.
\end{proposition}
\begin{proof}
Let $T:L^p(\hat G)\rightarrow A(G)$ be a bounded left $A(G)$-module homomorphism,
and suppose towards a contradiction that $T$ is not zero.  By density, we can
find $x\in L$, such that setting $\xi = j^*_{(0)}(x)$, we have that $T(\xi)\not=0$.
Let $a = \phi_*(\varphi_x) \in A(G)$.  For $y\in L$, let $\eta=j^*_{(0)}(y)$
and $b = \phi_*(\varphi_y)$.  Then, with reference to Theorem~\ref{thm:ag_mod_action},
$z = \hat\phi_*^{-1}(a)\cdot y\in L$ with $\phi_*(\varphi_z) = a \cdot \chi_*(\varphi_y)
= ab = ba = b \cdot \phi_*(\varphi_x) = \phi_*( \hat\phi_*^{-1}(b) \cdot x)$.  Thus
\[ a\cdot T(\eta) = Tj^*_{(0)}( z )
= Tj^*_{(0)}( \hat\phi_*^{-1}(b) \cdot x )
= b \cdot T(\xi). \]

Let $V$ be a compact neighbourhood of the identity in $G$, so that $0<|V|<\infty$.
Let $K$ be a compact neighbourhood of the identity with $KK^{-1} \subseteq V$,
let $r\in G$, let $\alpha = |K|^{-1/2} \chi_{r^{-1}K}\in L^2(G)$ and $\beta
= |K|^{-1/2} \chi_K \in L^2(G)$.  Then $\|\alpha\|_2 = \|\beta\|_2 = 1$, and so
$b = \omega_{\alpha,\beta}\in A(G)$ with $\|b\|_{A(G)}\leq 1$.  We see that
\[ b(s) = \frac{1}{|K|} \int \chi_{r^{-1}K}(s^{-1}t) \chi_K(t) \ dt
= \frac{|sr^{-1}K \cap K|}{|K|} \qquad (s\in G). \]
So $b(r)=1$ and $b(s)\not=0$ implies that $s \in KK^{-1}r \subseteq Vr$.  So
$b$ has compact support and is bounded, and hence $b\in L^1(G)$ with
$\|b\|_1 \leq |Vr|$.  By Proposition~\ref{prop:int_rrr_better},
$b = \phi_*(\varphi_y)$ where $y\in L$ with $y(f) = f\check b$ for $f\in C_{00}(G)$.
We can check that actually $y=\rho(\nabla^{-1/2}b)$, so that $\|y\| \leq
\| \nabla^{-1}b \|_1 \leq |Vr| \|\nabla^{-1}|_{Vr}\|_\infty = K(V)$ say.
By the estimate in Proposition~\ref{prop:limit} we see that
\[ \|j^*_{(0)}(y)\|_p \leq \|y\|^{1/p'} \|\varphi_y\|^{1/p} \leq K(V)^{1/p'}. \]
With $\eta = j^*_{(0)}(y)$, we hence see that
\[ |T(\xi)(r)| \leq  \| b \cdot T(\xi)\|_{A(G)} = \|a \cdot T(\eta)\|_{A(G)}
\leq \|a\|_{A(G)} \|T\| K(V)^{1/p'}. \]
In particular, we can make $K(V)$ as small as we like by choosing $V$ small
(as $G$ is not discrete).  As $r$ was arbitrary, we conclude that $T(\xi)=0$,
giving our contradiction.
\end{proof}

We can now follow the proof of \cite[Theorem~6.9]{fls} to show the following;
we refer the reader to \cite{fls} for the definition of \emph{operator projective}.

\begin{theorem}
Let $G$ be a non-discrete group and $1<p<\infty$.  Then $L^p(\hat G)$ is not operator
projective as a left $A(G)$-module.
\end{theorem}

\section{Representing the multiplier algebra}\label{sec:rep_ag}

Let $G$ be a locally compact group, let $(p_n)$ be a sequence in $(1,\infty)$ tending to $1$,
and let
\[ E = \ell^2-\bigoplus_n L^{p_n}(\hat G). \]
In the Banach space case, this is the direct sum in the $\ell^2$ sense, defined in
Section~\ref{sec:first}.  In the operator space case, we regard this as a discrete
vector-valued commutative $\ell^2$ space, which carries a natural operator space structure,
see \cite[Section~1]{xu} and \cite{pisier_non}.  Indeed, $E_\infty=\ell^\infty-\oplus
L^{p_n}(\hat G)$ carries an obvious operator space structure.  We give $E_1=\ell^1-\oplus
L^{p_n}(\hat G)$ the operator-space structure arising as a subspace of the dual of
$\ell^\infty-\oplus L^{p_n}(\hat G)^*$.  Then $(E_\infty,E_1)$ is a compatible couple,
and $E$ is simply $(E_\infty,E_1)_{[1/2]}$.  Notice that the underlying Banach space
is the same as the usual definition.

Then $A(G)$ acts co-ordinate wise on $E$, so that $E$ becomes a (completely)
contractive $A(G)$-module.  In the operator space case, notice that this is clear
for $E_1$ and $E_\infty$, and hence also for $E$ by bilinear interpolation.  In this
section, we shall show that $MA(G)$, respectively $M_{cb}A(G)$, have actions on $E$
extending those of $A(G)$, and that the resulting homomorphisms $MA(G)\rightarrow
\mc B(E)$ and $M_{cb}A(G)\rightarrow\mc{CB}(E)$ are weak$^*$-weak$^*$-continuous
(complete) isometries.

\begin{proposition}\label{prop:action_ma}
For $1<p<\infty$, there is a natural action of $MA(G)$ on $L^p(\hat G)$ extending the
action of $A(G)$, such that $a\cdot j^*_{(0)}\rho(f) = j^*_{(0)}\rho(a\cdot f)$ for
$a\in MA(G)$ and $f\in C_{00}(G)^2$.  Furthermore, this action of $MA(G)$ restricts
to give a completely contractive action of $M_{cb}A(G)$ on $L^p(\hat G)$.
\end{proposition}
\begin{proof}
We let $MA(G)$ act on $A(G)$ in the canonical way.  As in the proof of Lemma~\ref{lem:mod_struc},
we note that $MA_r(G)$ acts on $A_r(G)$ and hence on $VN_r(G)$ by duality.  This action
satisfies $a\cdot\rho(f) = \rho(a\cdot f)$ for $a\in MA(G)$ and $f\in C_{00}(G)$.  We then
extend $\hat\phi_*^{-1}$ to an isometric homomorphism $\psi:MA(G)\rightarrow MA_r(G)$, which
completes the argument as in Lemma~\ref{lem:mod_struc}.  We define $\psi$ by
\[ \psi(a)(b) = \hat\phi_*^{-1}\big( a \hat\phi_*(b) \big)
\qquad (a\in MA(G), b\in A_r(G)). \]
As $\hat\phi_*$ is a homomorphism, this does extend $\hat\phi_*^{-1}$ and is itself
a homomorphism.  Clearly $\psi$ is contractive, and has an obvious contractive inverse,
so that $\psi$ is isometric as required.  Notice that, if we view $a\in MA(G)$ and $\psi(a)$ as
functions on $G$ (using $\Phi$ and $\Phi'$) then these functions agree.

We now follow Theorem~\ref{thm:ag_mod_action} and use interpolation to extend this
$MA(G)$ action to $L^p(\hat G)$.  We hence need to show that if $x\in L$, then
$y=\psi(a)\cdot x\in L$ with $a\cdot \phi_*(\varphi_x) = \phi_*(\varphi_y)$.  As in
the proof of Theorem~\ref{thm:ag_mod_action}, by our approximation result, it is
enough to show this for $x=\rho(f^*g)$ for $f,g\in C_{00}(G)$.  But then the proof of
Theorem~\ref{thm:ag_mod_action} follows \emph{mutatis mutandis}.

The remark about $M_{cb}A(G)$ will follows if we can show that $\psi$ restricts to
a complete contraction $\psi:M_{cb}A(G)\rightarrow M_{cb}A_r(G)$.  However, this
follows immediately because $\hat\psi_*$ is a complete isometry.
\end{proof}

By \cite{DcH}, $MA(G)$ is a dual Banach algebra with a predual $Q$, which is the
completion of $L^1(G)$ for the norm
\[ \|f\|_Q = \sup\Big\{ \Big|\int_G f(s) a(s) \ ds\Big| : a\in MA(G), \|a\|\leq1 \Big\}
\qquad (f\in L^1(G)). \]
Let $\lambda_Q:L^1(G)\rightarrow Q$ be the inclusion map.
Similarly, $M_{cb}A(G)$ has a predual $Q_{cb}$ which is defined in the same way, but
taking the supremum over the unit ball of $M_{cb}A(G)$.  Define similarly
$\lambda_{Q_{cb}}:L^1(G)\rightarrow Q_{cb}$.

For $1<p<\infty$, let $\pi^p:A(G)\rightarrow\mc B(L^p(\hat G))$ be the contractive
homomorphism given by Theorem~\ref{thm:ag_mod_action}, and let $\hat\pi^p:MA(G)
\rightarrow\mc B(L^p(\hat G))$ be the contractive homomorphism given by
Proposition~\ref{prop:action_ma}.  Using Izumi's bilinear product, we have that
$L^p(\hat G)^* = L^{p'}(\hat G)$, and so we can consider the map
\[ \pi^p_*:L^p(\hat G) \proten L^{p'}(\hat G) \rightarrow A(G)^*=VN(G); \quad
\ip{\pi^p_*(\xi\otimes\eta)}{a} = \ip{a\cdot\xi}{\eta}_{p,(0)}, \]
for $a\in MA(G), \xi\in L^p(\hat G)$ and $\eta\in L^{p'}(\hat G))$.
Let $\hat\pi^p_*:MA(G)\rightarrow MA(G)^*$ be the analogous map.

Let $\pi^{p,cb}:A(G)\rightarrow\mc{CB}(L^p(\hat G))$ and
$\hat\pi^{p,cb}:M_{cb}A(G)\rightarrow\mc{CB}(L^p(\hat G))$ be analogously
given by Theorem~\ref{thm:ag_mod_action} and Proposition~\ref{prop:action_ma}.
Similarly, define $\pi^{p,cb}_*:L^p(\hat G)\proten L^{p'}(\hat G)\rightarrow A(G)^*$
and $\hat\pi^{p,cb}_*:L^p(\hat G)\proten L^{p'}(\hat G)\rightarrow MA(G)^*$.

\begin{proposition}\label{prop:pi_p_values}
The maps $\pi^p_*$ and $\pi^{p,cb}_*$ take values in $C^*_\lambda(G)$, the reduced
group C$^*$-algebra.  The map $\hat\pi^p_*$ takes values in the predual $Q$,
and $\hat\pi^{p,cb}_*$ takes values in the predual $Q_{cb}$, so that both $\hat\pi^p$
and $\hat\pi^{p,cb}$ are weak$^*$-weak$^*$-continuous.
\end{proposition}
\begin{proof}
Suppose that $\xi=j^*_{(0)}(f)$ and $\eta=j^*_{(0)}(g)$ for $f,g\in C_{00}(G)^2$.
Then, for $a\in MA(G)$, using the calculations of Lemma~\ref{lem:int_rrr},
\begin{align*} \ip{\hat\pi^p_*(\xi\otimes\eta)}{a}
&= \ip{j^*_{(0)}\rho(a\cdot f)}{j^*_{(0)}\rho(g)}_{p,(0)}
= \ip{\rho(a\cdot f)}{\varphi_{\rho(g)}} \\
&= \int_G a(s) f(s) \nabla(s)^{-1/2} g(s^{-1}) \ ds
= \ip{a}{\lambda_Q(f\cdot Kg)}. \end{align*}
Hence $\hat\pi^p_*(\xi\otimes\eta) = \lambda_Q(f\cdot Kg) \in Q$.
By Proposition~\ref{prop:den_in_lp}, such $\xi$ and $\eta$ are norm dense, showing
that $\hat\pi^p_*$ takes values in $Q$.  It is now standard that $\hat\pi^p$ is
weak$^*$-weak$^*$-continuous.  The same calculation shows that
$\hat\pi^{p,cb}_*(\xi\otimes\eta) = \lambda_{Q_{cb}}(f\cdot Kg) \in Q_{cb}$,
so that $\hat\pi^{p,cb}_*$ takes values in $Q_{cb}$ and hence also
$\hat\pi^{p,cb}$ is weak$^*$-weak$^*$-continuous.

We have that $\hat\pi^p$, restricted to $A(G)$, is $\pi^p$.  Similarly, and for
$f\in L^1(G)$, we see that $\lambda_Q(f)$, restricted to $A(G)$, is simply
$\lambda(f)\in C^*_\lambda(G) \subseteq VN(G)$.  The above calculation hence
also shows that $\pi^p_*$ takes values in $C^*_\lambda(G)$, as claimed.
The same argument applies in the completely bounded case.
\end{proof}

If $\mc A$ is a commutative Banach algebra and $(L,R)\in M(\mc A)$ then for $a,b\in\mc A$,
$L(a)b = L(ab) = L(ba) = L(b)a = aL(b) = R(a)b$.  If $\mc A$ is faithful, then $L=R$.
We remark that $A(G)$ is faithful, as by \cite[Lemme~3.2]{eymard}, for any compact $K\subseteq
G$ there exists $a\in A(G)$ which is identically $1$ on $K$.

The following is now the $A(G)$ version of the results in Section~\ref{sec:first}.

\begin{theorem}
Let $G$ and $E$ be as above.  Let $MA(G)$ act on $E$ co-ordinate wise.  Then the
resulting homomorphism $\pi:MA(G)\rightarrow\mc B(E)$ is an isometry, and is
weak$^*$-weak$^*$-continuous.  Furthermore,
the image of $\pi$ is the idealiser of $\pi(A(G))$ in $\mc B(E)$.
\end{theorem}
\begin{proof}
Clearly $\pi$ is contractive.  Let $a\in MA(G)$ and $\epsilon>0$.  As $j_{(0)}(L)$
is dense in $A_r(G)$, we can find $x\in L$ with $\|\varphi_x\|=1$ and
\[ \| a \cdot \phi_*(\varphi_x) \| \geq (1-\epsilon)\|a\|_{MA(G)}. \]
Then, using Proposition~\ref{prop:limit}, we see that
\[ \|\pi(a)\| \geq \lim_n \|a \cdot j^*_{(0)}(x) \|_{p_n} \|j^*_{(0)}(x) \|_{p_n}^{-1}
= \| a \cdot \phi_*(\varphi_x) \| \|\varphi_x\| \geq (1-\epsilon) \|a\|_{MA(G)}. \]
As $\epsilon>0$, we conclude that $\pi$ is an isometry, as required.

Let $\xi=(\xi_n)\in E$ and $\eta=(\eta_n)\in E^*$ be sequences which are eventually zero.
For $a\in MA(G)$, we see that
\[ \ip{\pi(a)\xi}{\eta} = \sum_n \ip{a}{\hat\pi^p_*(\xi_n\otimes\eta_n)}, \]
so that $\pi_*(\xi\otimes\eta) \in Q$.  As such $\xi$ and $\eta$ are dense, by continuity
we see that $\pi_*:E\proten E^*\rightarrow MA(G)^*$ takes values in $Q$.  Again, this
implies that $\pi$ is weak$^*$-weak$^*$-continuous.

Clearly $\pi(MA(G))$ is contained in the idealiser of $\pi(A(G))$.  Conversely,
given $T$ in the idealiser of $\pi(A(G))$, we can follow the proof of
Theorem~\ref{thm:ideal_meas} to find $a\in MA(G)$ with
$\pi(ab) = T \pi(b)$ and $\pi(ba) = \pi(b) T$ for $b\in A(G)$.  For each $L^p(\hat G)$,
by Proposition~\ref{prop:den_in_lp} and again using \cite[Lemme~3.2]{eymard},
it follows that $\{ \pi(a)\xi : a\in A(G), \xi\in L^p(\hat G) \}$
is linearly dense in $L^p(\hat G)$.  This is enough to show that then $T=\pi(a)$
as required to complete the proof.
\end{proof}

The completely bounded version of this result requires a subtly different proof.

\begin{theorem}\label{thm:mcb_rep}
Let $G$ and $E$ be as above, where we now regard $E$ as an operator space.
Let $M_{cb}A(G)$ act on $E$ co-ordinate wise.  Then the
resulting homomorphism $\pi_{cb}:M_{cb}A(G)\rightarrow\mc{CB}(E)$ is a
weak$^*$-weak$^*$-continuous complete isometry.  Furthermore, the image of
$\pi_{cb}$ is the idealiser of $\pi_{cb}(A(G))$ in $\mc{CB}(E)$.
\end{theorem}
\begin{proof}
Again, clearly $\pi_{cb}$ is completely contractive.  As the norm on
$\mathbb M_n(L^p(\hat G))$ is given by interpolating $\mathbb M_n(VN_r(G))$ and
$\mathbb M_n(A(G))$, we can simply apply the proof of the previous theorem, but working
with matrices, and using Proposition~\ref{prop:limitcb}, to show that $\pi_{cb}$ is a
complete isometry.  Similarly, it follows that $\pi_{cb}$ is weak$^*$-weak$^*$-continuous.

Clearly $\pi_{cb}(M_{cb}A(G))$ is contained in the idealiser of $\pi_{cb}(A(G))$.
Conversely, given $T$ in the idealiser of $\pi_{cb}(A(G))$, we can follow the proof of
Theorem~\ref{thm:ideal_meas} to find $(L,R)\in M(A(G))$ with $\pi_{cb}(L(a)) =
T\pi_{cb}(a)$ and $\pi_{cb}(R(a)) = \pi_{cb}(a)T$ for $a\in A(G)$.

For $n\in\mathbb N$, let $i_n:L^{p_n}(\hat G)\rightarrow E$ be the inclusion map,
which is a completely contractive $A(G)$-bimodule homomorphism.  Then
\[ Ti_n(a\cdot\xi) = T \pi_{cb}(a) i_n(\xi) = \pi_{cb}(L(a)) i_n(\xi)
= i_n(L(a)\cdot\xi)
\qquad (a\in A(G), \xi\in L^{p_n}(\hat G)). \]
As $A(G)\cdot L^p(\hat G)$ is dense in $L^p(\hat G)$ for all $p$, we conclude that
there exists $T_n\in\mc{CB}(L^{p_n}(\hat G))$ with $Ti_n = i_nT_n$ and
$\|T_n\|_{cb} \leq \|T\|_{cb}$.  It now follows that
\[ T_n(a\cdot\xi) = L(a)\cdot\xi, \quad a \cdot T_n(\xi) = R(a)\cdot\xi
\qquad (a\in A(G), \xi\in L^{p_n}(\hat G)). \]

Let $A_0 = A(G)\cap VN_r(G)$ regarded as a subspace of $A(G)$ (so that $A_0$
is $\phi_* j_{(0)}(L)$).  Consider the map $i^*_{(0)}\phi_*^{-1}:A(G) \rightarrow
L^*_{(0)}$, which maps $A_0$ into $L^p(\hat G)$ for all $p$.  Let
$\iota_n: A_0 \rightarrow L^{p_n}(\hat G)$ be the resulting map.  We have that
$a\cdot\iota_n(b) = \iota_n(ab)$ for $a\in A(G)$ and $b\in A_0$.
We hence see that
\[ a\cdot T_n \iota_n(b) = R(a)\cdot\iota_n(b) = \iota_n(R(a)b) = \iota_n(a L(b))
= a \cdot \iota_n(L(b)) \qquad (a\in A(G), b\in A_0). \]
It follows that $T_n \iota_n = \iota_n L$.  By much the same argument as at the
start of the proof, we see that for $a=(a_{ij})\in \mathbb M_m(A_0)$,
\[ \|(L)_m(a)\| = \lim_n \| (\iota_n L(a_{ij})) \| = \lim_n \|(T_n \iota_n(a_{ij})) \|
\leq \|T\|_{cb} \lim_n \|(\iota_n(a_{ij}))\| \leq \|T\|_{cb} \|(a_{ij})\|. \]
Thus $L$ is completely bounded, with $\|L\|_{cb}\leq\|T\|_{cb}$,
and so induces a member of $M_{cb}A(G)$.
We can now follow the end of the previous proof to conclude that $T \in
\hat\pi_{cb}(M_{cb}A(G))$.
\end{proof}

\section{Analogues of the Figa-Talamanca--Herz algebras}

In Section~\ref{sec:first} we saw that the Figa-Talamanca--Herz algebras $A_p(G)$
naturally appeared.  We have now developed enough theory to very easily suggest a
definition for analogues of the Figa-Talamanca--Herz algebras, starting with $A(G)$
instead of $L^1(G)$.  Indeed, consider the map $\pi^p:A(G)\rightarrow\mc B(L^p(\hat G))$
as in the previous section.  We define $A_p(\hat G)$ to be the image of $\pi^p_*$,
equipped with the quotient norm, so that $A_p(\hat G)$ is isometric to
$(L^p(\hat G) \proten L^{p'}(\hat G)) /\ker \pi^p_*$.  By Proposition~\ref{prop:pi_p_values}
we see that $A_p(\hat G)$ is a subspace of $C^*_\lambda(G)$, which we would
expect, as this is the ``dual'' statement to the fact that $A_p(G)\subseteq C_0(G)$.

The following says, informally, that $A_2(\hat G) = L^1(G)$.

\begin{theorem}\label{thm:atwo_is_gpconv}
For a locally compact group $G$, $A_2(\hat G)$ is equal to $\lambda(L^1(G))$ as
a subset of $C^*_\lambda(G)$, and the norm on $A_2(\hat G)$ agrees with that on
$L^1(G)$.
\end{theorem}
\begin{proof}
We recall the isometric isomorphism $\theta:L^2(G)\rightarrow L^2(\hat G)$ given
by Proposition~\ref{prop:std_hilbert}, $\theta(f) = j^*_{(0)}\rho(\Delta^{-3/4}\check f)$
for $f\in C_{00}(G)^2$.  Then, from above,
\[ \pi^2_*(\theta(f)\otimes\theta(g)) = \lambda\big( \Delta^{-3/4}\check f
   \cdot K(\Delta^{-3/4}\check g) \big)
= \lambda\big( \Delta^{-1/2} \check f \cdot g \big)
= \lambda\big( Kf \cdot g \big) \qquad (f,g\in C_{00}(G)^2). \]
As $K:L^2(G)\rightarrow L^2(G)$ is unitary, by continuity,
we have that $\pi^2_*(\theta(\xi)\otimes\theta(\eta))
= \lambda(K\xi\cdot\eta)$ for $\xi,\eta\in L^2(G)$.  In particular, by Cauchy-Schwarz,
we have that $K\xi\cdot\eta \in L^1(G)$ with $\|K\xi\cdot\eta\|_1 \leq \|K\xi\|_2\|\eta\|_2$,
for $\xi,\eta\in L^2(G)$.

For $\tau\in L^2(G)\proten L^2(G)$ and $\epsilon>0$, we can find sequences $(\xi_n)$
and $(\eta_n)$ in $L^2(G)$ with
\[ \tau = \sum_n \xi_n\otimes\tau_n, \quad \|\tau\| \leq \sum_n \|\xi_n\|_2\|\eta_n\|_2
< \|\tau\|+\epsilon. \]
Then let $f=\sum_n K\xi_n\cdot\eta_n \in L^1(G)$, the sum converging by Cauchy-Schwarz,
with $\|f\|_1 \leq \|\tau\|+\epsilon$.  We see that
\[ \pi^2_*(\theta\otimes\theta)\tau = \lambda\Big( \sum_n K\xi_n \cdot \eta_n \Big)
= \lambda(f). \]
As $(\theta\otimes\theta)$ is an isometric isomorphism, it follows that $A_2(\hat G) \subseteq
\lambda(L^1(G))$.

For $f\in L^1(G)$, let $\xi = K(|f|^{1/2})\in L^2(G)$ and $\eta = f |f|^{-1/2} \in L^2(G)$,
so that $\pi^2_*(\theta(\xi)\otimes\theta(\eta)) = f$, and $\|\xi\|_2 = \|\eta\|_2 = \|f\|_1^{1/2}$.
We conclude that $A_2(\hat G) = \lambda(L^1(G))$, with the quotient norm on $A_2(\hat G)$
agreeing with the $L^1$ norm on $\lambda(L^1(G))$.
\end{proof}

In particular, $A_2(\hat G)$ is a subalgebra of $C^*_\lambda(G)$, and with the quotient norm,
$A_2(\hat G)$ is a Banach algebra.  We have been unable to
decide if the same is true for $A_p(\hat G)$, for $p\not=2$.  However, we do have the following.

\begin{proposition}\label{prop:aphatg_ds}
For $1<p<\infty$, $A_p(\hat G)$ contains a dense subset which is a subalgebra of
$C^*_\lambda(G)$.
\end{proposition}
\begin{proof}
Let $a,b,c,d\in C_{00}(G)^2$, let $\xi_1=j^*_{(0)}(a), \xi_2=j^*_{(0)}(c) \in L^p(\hat G)$
and let $\eta_1=j^*_{(0)}(b), \eta_2=j^*_{(0)}(d) \in L^{p'}(\hat G)$.  Then,
as above, $\pi^p_*(\xi_1\otimes\eta_1) = \lambda(a\cdot Kb)$ and
$\pi^p_*(\xi_2\otimes\eta_2) = \lambda(c\cdot Kd)$.  Let $f=(a\cdot Kb)(c\cdot Kd)
\in C_{00}(G)^2$.

Pick $g_1\in C_{00}(G)$ with $\int_G g_1(s) \ ds =1$.
Let $X \subseteq G$ be a compact set containing the support of $f$, and let
$Y\subseteq G$ be a compact set containing the support of $g_1$.
Let $e=|Y|^{-1}\chi_{(XY)^{-1}Y}$ and $f=\chi_Y$, so that
$g_0 = e\check f \in C_{00}(G)$.  Then, for $s\in G$,
\[ g_0(s) = \int_G e(t) \check f(t^{-1}s) \ dt
= \frac{1}{|Y|} \int_{(XY)^{-1}Y} \chi_Y(s^{-1}t) \ dt
= \frac{|sY \cap (XY)^{-1}Y|}{|Y|}, \]
so if $s\in (XY)^{-1}$, then $sY \subseteq (XY)^{-1}Y$ and so
$g_0(s) = |sY||Y|^{-1} = 1$.  Now let 
$g = (\nabla^{-1/2}g_1)(\nabla^{-1/2}g_0) \in C_{00}(G)^2$, so for $s\in X$,
\begin{align*} (\nabla^{-1/2}g_1) & (\nabla^{-1/2}g_0)(s^{-1})
= \int_G \nabla(t)^{-1/2} g_1(t) \nabla(t^{-1}s^{-1})^{-1/2}  g_0(t^{-1}s^{-1}) \ dt \\
&= \nabla(s)^{1/2} \int_Y g_1(t) g_0(t^{-1}s^{-1}) \ dt
= \nabla(s)^{1/2} \int_Y g_1(t) \ dt = \nabla(s)^{1/2},
\end{align*}
as if $t\in Y$ then $t^{-1}s^{-1} \in (XY)^{-1}$.  Hence, for $s\in X$,
we see that $Kg(s) = g(s^{-1}) \nabla(s)^{-1/2} = 1$.
Thus $f \cdot Kg = f$, showing that
\[ \pi^p_*(\xi_1\otimes\eta_1) \pi^p_*(\xi_2\otimes\eta_2)
= \lambda(f) = \pi^p_*(j^*_{(0)}\rho(f)\otimes j^*_{(0)}\rho(g)). \]

We conclude that
\[ \lin \big\{ \pi^p_*(j^*_{(0)}\rho(f)\otimes j^*_{(0)}\rho(g)) :
f,g\in C_{00}(G)^2 \big\} \subseteq A_p(\hat G) \]
is a dense subalgebra.
\end{proof}

One could instead work with $\pi^{p,cb}_*$, which would lead to an operator
space version of $A_p(\hat G)$, say $OA_p(\hat G)$.  However, as this would
naturally use the operator space projective tensor product, in general we would
only have that $A_p(\hat G) \subseteq OA_p(\hat G)$.  Indeed, in
\cite{runde}, Runde used the natural operator space structure on vector valued
\emph{commutative} $L^p$ spaces to define algebras $OA_p(G)$, as an attempt to find
an operator space structure on $A_p(G)$.  If $G$ is abelian, then by the Fourier
transform, $OA_p(\hat G)$ has an unambiguous meaning (either ours or Runde's).
Let $PM_p(\hat G)$ be the weak$^*$-closure of $\pi^p(A(G))$ in
$\mc B(L^p(\hat G))$.  After \cite[Proposition~2.1]{runde}, in a remark
attributed to G. Pisier, it is shown that there exist abelian $G$ with
$PM_p(\hat G) \not\subseteq \mc{CB}(L^p(\hat G))$.  It follows that $OA_p(\hat G)
\not= A_p(\hat G)$.  If we wish to view $OA_p(\hat G)$ as a generalisation of
$A_p(G)$, then this a problem!

\bigskip
\noindent\textbf{Author's address:}
\parbox[t]{5in}{School of Mathematics,\\
University of Leeds,\\
Leeds LS2 9JT\\
United Kingdom}

\smallskip
\noindent\textbf{Email:} \texttt{matt.daws@cantab.net}

\end{document}